\documentclass[12pt,reqno]{amsart}
\usepackage{amssymb,amsmath,amsthm,enumerate,verbatim}
\usepackage[utf8]{inputenc} 
\usepackage[english]{babel} 
\usepackage{mathrsfs}
\usepackage{enumitem}
\usepackage[margin=1.0in]{geometry}
\usepackage[dvipsnames,usenames]{color}
\DeclareMathOperator{\Ker}{Ker}

\DeclareMathOperator{\Sym}{Sym}

\DeclareMathOperator{\ord}{ord}

\renewcommand\Re{\hbox{{\rm Re}}\,}
\newcommand{\abs}[1]{\lvert#1\rvert}
\newcommand{\Abs}[1]{\left\lvert#1\right\rvert}
\newcommand{\norm}[1]{\lVert#1\rVert}

\newcommand{\jap}[1]{\langle#1\rangle}
\newcommand{\hel}{\alpha}
\newcommand{\han}{\beta}

\newcommand{\bD}{{\mathbf D}}
\newcommand{\bc}{{\mathbf c}}
\newcommand{\ba}{{\mathbf a}}
\newcommand{\bbT}{{\mathbb T}}
\newcommand{\bbR}{{\mathbb R}}
\newcommand{\bbC}{{\mathbb C}}

\newcommand{\bbN}{{\mathbb N}}

\newcommand{\bbD}{{\mathbb D}}

\newcommand{\calN}{\mathcal{N}}

\newcommand{\scrB}{\mathscr{B}}

\numberwithin{equation}{section}

\theoremstyle{plain}
\newtheorem{theorem}{\bf Theorem}[section]
\newtheorem*{theorem*}{Theorem 1.1$'$}
\newtheorem{lemma}[theorem]{\bf Lemma}
\newtheorem{proposition}[theorem]{\bf Proposition}

\newtheorem{corollary}[theorem]{\bf Corollary}

\theoremstyle{definition}

\theoremstyle{remark}
\newtheorem*{remark*}{\bf Remark}
\newtheorem{remark}[theorem]{\bf Remark}
\newtheorem{example}[theorem]{\bf Example}
\newtheorem{question}[theorem]{\bf Question}

\newcommand{\wt}{\widetilde}
\newcommand{\eps}{\varepsilon}

\renewcommand{\[}{\begin{equation}}
\renewcommand{\]}{\end{equation}}

\DeclareMathOperator{\dom}{dom}

\begin{document} 
\title[Helson matrices]{On Helson matrices: moment problems, non-negativity, boundedness, and finite rank}
\date{\today} 

\author{Karl-Mikael Perfekt}
\address{Department of Mathematics and Statistics, 
  University of Reading, Reading RG6 6AX, United Kingdom}
\email{k.perfekt@reading.ac.uk}

\author{Alexander Pushnitski} 
\address{Department of Mathematics,
King's College London,
Strand, London WC2R 2LS,
United Kingdom}
\email{alexander.pushnitski@kcl.ac.uk}


\begin{abstract}
We study Helson matrices (also known as multiplicative Hankel matrices), 
i.e. infinite matrices of the form $M(\hel) = \{\hel(nm)\}_{n,m=1}^\infty$, 
where $\hel$ is a sequence of complex numbers. 
Helson matrices are considered as linear operators on $\ell^2(\bbN)$. 
By interpreting Helson matrices as Hankel matrices in countably many variables we use the theory of multivariate moment problems to show that $M(\hel)$ is non-negative if and only if $\hel$ is the moment sequence of a measure $\mu$ on $\bbR^\infty$, assuming that $\hel$ does not grow too fast. 
We then characterize the non-negative bounded Helson matrices $M(\hel)$ as those where the corresponding  
moment measures $\mu$ are Carleson measures for the Hardy space of countably many variables. 
Finally, we give a complete description of the Helson matrices of finite rank, in parallel with the classical Kronecker theorem on Hankel matrices.
\end{abstract}

\maketitle
\section{Introduction}
\subsection{Overview}

Let $\{\hel(n)\}_{n=1}^\infty$ be a sequence of complex numbers.
We denote by $M(\hel)$ the (potentially unbounded) operator 
on $\ell^2(\bbN)$ with ``matrix entries'' $\{\hel(nm)\}_{n,m=1}^\infty$, 
i.e. the  $(n,m)$'th entry depends only on the product $nm$.  
Following \cite{Queffelec},
we will call $M(\hel)$ the \emph{Helson matrix}
corresponding to the sequence $\hel$, 
in honor of Henry Helson's initial investigations \cite{Helson06} and \cite{Helson10}. 

Recall that for a sequence of complex numbers $\{\han(j)\}_{j=0}^\infty$, the corresponding \emph{Hankel matrix}
is the operator on $\ell^2(\bbN_0)$ ($\bbN_0=\{0,1,2,\dots\}$) 
with the matrix entries $\{\han(j+k)\}_{j,k=0}^\infty$. We denote this 
Hankel matrix by $H(\han)$. 

Our aim with this paper is to compare the theory of Helson matrices, a theory in its infancy, 
with the well-established theory of Hankel matrices. 
More precisely, we will address the following questions for Helson matrices:
\begin{enumerate}[label=(\alph*)]
\item \label{q:bounded}
What is the class of all bounded $M(\hel)$?
\item \label{q:nonneg}
What is the class of all non-negative $M(\hel)$?
\item \label{q:finiterank}
What is the class of all $M(\hel)$ of finite rank?
\item \label{q:spectra}
Can one describe the spectra of $M(\hel)$ for some concrete sequences $\hel$?
\end{enumerate}
These questions are well understood for Hankel matrices $H(\han)$. 
In this context,
Question~\ref{q:bounded} is answered by Nehari's theorem, 
Question~\ref{q:nonneg} is answered by the solution to the Hamburger moment problem,
and Question~\ref{q:finiterank} is answered by Kronecker's theorem.  
As for Question~\ref{q:spectra}, many of its answers originate from the description of the spectrum of the Hilbert matrix 
$\{1/(1+j+k)\}_{j,k=0}^\infty$ given by M.~Rosenblum \cite{Rosenblum1,Rosenblum2}.
More details of these classical facts will appear throughout the paper, as our investigation of Helson matrices is guided by the analogy with Hankel matrices, and we have tried to make this analogy as explicit as possible.

To summarize the main findings of the paper, we will fully answer Question~\ref{q:finiterank} with a statement analogous to the classical Kronecker theorem. We will observe that Question~\ref{q:nonneg} is almost settled by the existing theory of multivariate moment problems, showing that when $\hel$ is polynomially bounded, then 
$M(\hel)$ is non-negative if and only if $\hel$ is the moment sequence of a positive measure on $\bbR^\infty$. 
To treat Question~\ref{q:bounded} we use this description of $\alpha$ to characterize the bounded \textit{non-negative} Helson matrices in terms of Carleson measures for the Hardy space $H^2(\bbT^\infty)$ of the infinite polytorus. However, such Carleson measures are notoriously difficult to understand, and we leave the topic with many remaining question marks. 
We will also discuss the spectral structure of a particular Helson matrix, known as the \emph{multiplicative Hilbert matrix}; 
this gives an answer to Question~\ref{q:spectra}.

\subsection{Helson matrices as Hankel matrices in infinitely many variables} \label{sec:prelim}
Observe that by restricting a Helson matrix $M(\hel)$ to indices of the form $n = p^j$ 
for some fixed prime $p$ and $j=0,1,\ldots$, one obtains a Hankel matrix $\{\hel(p^{j+k})\}_{j,k=0}^\infty$. 
To develop this further, we need some standard notation. Let $p = \{p_j\}_{j=1}^\infty$ be the monotone enumeration of the primes.
Every natural number $n$ can be written $n=p^\kappa$, where
 $\kappa=(\kappa_1,\kappa_2,\dots)$ is a  multi-index with all components $\kappa_j\in\bbN_0$ and 
\[
p^\kappa=\prod_{j=1}^\infty p_j^{\kappa_j}.
\label{a1}
\]
In this parameterisation $\kappa$ runs through the subset of 
$\bbN_0^\infty=\bbN_0\times\bbN_0\times\dots$
which consists of finite sequences: $\kappa_j=0$ for all sufficiently large $j$.
We will denote this subset by $\bbN_{0}^{(\infty)}$. 
Formula \eqref{a1} establishes a bijection between $\bbN$ and $\bbN_{0}^{(\infty)}$. 
We will sometimes write $\kappa=\kappa(n)$ or $n=n(\kappa)$ to express 
the relationship $n=p^\kappa$, 
allowing us to alternatively index the sequence $\{\hel(n)\}_{n=1}^\infty$ by $\kappa$.
Obviously, we can also view $\ell^2(\bbN)$ as $\ell^2(\bbN_{0}^{(\infty)})$.

Now let $n=p^\kappa$ and $m=p^{\kappa'}$ for two multi-indices $\kappa,\kappa'\in\bbN_{0}^{(\infty)}$. Then we have that
\[
\hel(nm)=\hel(p^\kappa p^{\kappa'})=\hel(p^{\kappa+\kappa'}), \quad \kappa,\kappa'\in\bbN_{0}^{(\infty)},
\label{a2}
\]
and so the Helson matrix $M(\hel)$ can be viewed as a Hankel matrix on infinitely many variables, 
i.e. as a Hankel matrix acting on $\ell^2(\bbN_{0}^{(\infty)})$. 
To make this clearer, let us consider an intermediate model, namely, Hankel matrices on finitely many variables.
Fix $d\in\bbN$ and let $\{\han(\kappa)\}_{\kappa\in\bbN_0^d}$ be a sequence of complex numbers. 
Then the Hankel matrix $H_d(\han)$ on $d$ variables is the linear operator
on $\ell^2(\bbN_0^d)$ formally given by 
$$
(H_d(\han)a)_\kappa=\sum_{\kappa'\in\bbN_0^d}\han(\kappa+\kappa')a_{\kappa'}, \quad \kappa\in\bbN_0^d,
$$
where $a=\{a_\kappa\}_{\kappa\in\bbN_0^d}\in\ell^2(\bbN_0^d)$. 
Thus, \eqref{a2} establishes a one-to-one correspondence between $M(\hel)$ and the analogue of $H_d(\han)$ for countably many variables ($d = \infty$).

\subsection{Hardy spaces} \label{sec:prelimhardy}

To treat Hankel and Helson matrices, we will make use of the Hardy spaces of the finite polytori $\bbT^d$, $d\in\bbN$, and of 
the countably infinite polytorus $\bbT^\infty=\prod_{j=1}^\infty \bbT$. 
For a variable $z = (z_1, z_2, \ldots)\in\bbT^d$ and for a multi-index $\kappa\in\bbN_0^d$
($\kappa\in\bbN_{0}^{(\infty)}$ if $d=\infty$) we use the same shorthand as in \eqref{a1}: 
$$
z^\kappa=\prod_{j} z_j^{\kappa_j}. 
$$ 
The Hardy space $H^2(\bbT^d)$  
consists of power series 
$f(z) = \sum_{\kappa} f_\kappa z^\kappa$ with the finite norm
$$
\|f\|_{H^2(\mathbb{T}^d)}^2 = \sum_{\kappa} | f_\kappa|^2 < \infty. 
$$
Here  the summation is over   $\kappa\in\bbN_0^d$ if $d<\infty$ and 
over $\kappa\in\bbN_{0}^{(\infty)}$ if $d=\infty$. 
Clearly, 
$H^2(\mathbb{T}^d)$ is a Hilbert space with the standard inner product
$$
\jap{f,g}_{H^2(\bbT^d)}=\sum_{\kappa} f_\kappa\overline{g_\kappa}.
$$
It can be naturally understood as a subspace of $L^2(\bbT^d)$, the space of square-integrable functions 
with respect to the Haar measure $dm_d$ on $\bbT^d$ \cite{CG86, HLS}. 
It is the subspace of $L^2(\bbT^d)$-functions whose
Fourier coefficients 
$f_\kappa$, $\kappa = (\kappa_1, \kappa_2, \ldots)$, are non-zero only if $\kappa_j \geq 0$ for every $j$. 
For this reason, $H^2(\mathbb{T}^\infty)$ is sometimes also known as the narrow cone.

An important fact is that each $f \in H^2(\mathbb{T}^d)$, $d<\infty$, defines a holomorphic function 
on $\mathbb{D}^d$. 
The reproducing kernel at the point $w \in \mathbb{D}^d$ is given by
\[ \label{eq:kdkernel}
K_d(z,w) = \prod_{j=1}^d \frac{1}{1-z_j \overline{w}_j} = 
\sum_{\kappa} z^\kappa \overline{w}^\kappa, \quad z \in \mathbb{D}^d.
\]
For $d=\infty$, in order for the holomorphic extension to make sense one must consider
it on  $\mathbb{D}^\infty \cap \ell^2(\bbN)$ \cite{CG86}. 
The reproducing kernel at the point $w \in \mathbb{D}^\infty \cap \ell^2(\bbN)$ is given by the same formula,
$$
K_\infty(z,w) = \prod_{j=1}^\infty \frac{1}{1-z_j \overline{w}_j} 
= 
\sum_{\kappa} z^\kappa \overline{w}^\kappa, \quad z \in \mathbb{D}^\infty \cap \ell^2(\bbN).
$$
A third name for $H^2(\bbT^\infty)$ is hence the Hardy space of the infinite polydisc.

\subsection{The structure of the paper} 
In Section~\ref{sec:bddnehari}, which contains no new results, we recall some background information
on the boundedness of Hankel and Helson matrices. We recall Nehari's description \cite{Nehari} of bounded Hankel matrices in one variable as well as the deep extension of this theorem to the multivariate setting due to Ferguson--Lacey--Terwilleger \cite{FergusonLacey, LaceyTerwilleger}. We also recall that for Helson matrices, the natural analogue of Nehari's theorem fails, as was demonstrated by Ortega-Cerd\`a and Seip \cite{OS}.

In Section~\ref{sec:nonneg} we discuss the connection of Hankel and Helson matrices to the 
theory of multivariate moment problems. 
The solution to the classical Hamburger moment problem says that the one-variable
Hankel matrix $H_1(\han)$ is non-negative (in the quadratic form sense)
if and only if it can be represented as the moment sequence of a positive measure on $\bbR$:
$$
\han(\kappa)=\int_\bbR t^\kappa d\mu(t), \quad \kappa\in\bbN_0.
$$
In Section~\ref{sec:nonneg}, we discuss the known multi-variable analogue of this statement 
for Hankel matrices $H_d(\han)$, $d<\infty$; in this case, $\beta$ is the moment sequence 
of a measure on $\bbR^d$. 
Further, under some growth 
restriction on $\hel$, we will give the (essentially known) result for Helson matrices: $M(\hel)$ is 
non-negative if and only if $\hel$ is a moment sequence for some measure $\mu$ on $\bbR^\infty$:
$$
\hel(n)=\int_{\bbR^\infty}t^\kappa d\mu(t), \quad n=p^\kappa. 
$$

In Section~\ref{sec:bounded} we use the moment sequence description of $\hel$ to characterize bounded Helson matrices in the special case when $M(\hel)$ is non-negative. In this case, we show that $M(\hel)$ is bounded if and only if the moment measure $\mu$ of $\hel$ is a Carleson measure for the Hardy space $H^2(\bbT^\infty)$, and this holds if and only if a certain integral operator is bounded on $L^2(\mu)$. Our result parallels Widom's \cite{Widom66} description of bounded non-negative Hankel matrices in a single variable. A notable difference is that the single variable problem has a natural extremal solution, namely, the Hilbert matrix. No such extremal exists for non-negative Helson matrices.
  
 In Section~\ref{sec:bohr} we consider a natural subclass of non-negative Helson matrices, inspired by the customary identification (the Bohr lift) between $H^2(\bbT^\infty)$ and the Hardy space $\mathscr{H}^2$ of Dirichlet series \cite{HLS}. For this subclass the description of boundedness can be taken further. 
In this context, we also discuss the spectral analysis of the \emph{multiplicative Hilbert matrix} 
\begin{equation}
\mathbf{H} \colon   \ell^2(\bbN_2) \to \ell^2(\bbN_2), \quad\bbN_2=\{2,3,4,\dots\},
\qquad 
\mathbf{H} = \left\{\frac{1}{\sqrt{nm}\log(nm)}\right\}_{n,m=2}^\infty,
\label{multH}
\end{equation}
which fits into this framework. 
A partial description of the spectrum of $\mathbf{H}$ was given in \cite{BPSSV}.
Here we report on further progress in this direction, while the details are given in a separate 
publication \cite{PP2}.  

In Sections~\ref{sec:finiterank} and \ref{sec:frproof} we characterize the finite rank Helson matrices. The material of Section~\ref{sec:nonneg} yields a description of the non-negative finite rank Helson matrices, but the general situation is significantly more difficult to handle. We are inspired by earlier results \cite{Pow82,JPR,Rochberg95}, but our setting requires further algebraic and analytic analysis.
As an algebraic question, the characterization of finite rank is asking for the study of finite-codimensional ideals in the commutative ring $\bbC[z_1, z_2, \ldots]$ of polynomials in countably many variables. 
The material of Sections~\ref{sec:finiterank} and \ref{sec:frproof} is hence based on basic commutative algebra to identify the correct algebraic description of finite rank Helson matrices. We then use the theory of $H^2(\bbT^\infty)$ to determine which of these yield densely defined (and hence bounded) Helson matrices. Loosely speaking, our result says that finite rank Helson matrices are generated by directional derivatives in a finite number of directions in $\bbC^\infty$  and by point evaluations. 

In the Appendix we study basic properties of measures $\mu$ on $\bbR^\infty$ which are Carleson for $H^2(\bbT^\infty)$ --- i.e. the measures that generate bounded non-negative Helson matrices. In addition to technical material needed for Section~\ref{sec:bounded}, we also give a proof that the natural reproducing kernel testing condition is insufficient to characterize the Carleson measures of this type.

\subsection{Acknowledgments}
We give special mention to Steffen Oppermann for providing and clarifying many of the algebraic details in Section~\ref{sec:finiterank}. We are also grateful to Ole Fredrik Brevig, William T. Sanders, and Herv\'e Queff\'elec for helpful discussions.

\section{Boundedness: Nehari's theorem} \label{sec:bddnehari}

\subsection{Bounded Hankel matrices} \label{sec:finitenehari}
First we recall Nehari's theorem in the finite variable setting. 
For $d\in\bbN$, let $\{\han(\kappa)\}_{\kappa\in\bbN_0^d}$ be a sequence of complex numbers
such that 
$$
b(z)=\sum_{\kappa\in\bbN_0^d} \han(\kappa) z^\kappa\in H^2(\bbT^d). 
$$
Consider the $d$-variable Hankel matrix $H_d(\han)$ on $\ell^2(\bbN_0^d)$; 
in this context $b$ is known as the \emph{analytic symbol} of $H(\han)$. 
For any polynomials of the form
$$
f(z)=\sum_{\kappa\in\bbN_0^d} \overline{f_\kappa} z^\kappa, \quad g(z)=\sum_{\kappa\in\bbN_0^d} g_\kappa z^\kappa,
$$
Parseval's theorem shows that
$$
\jap{H_d(\han)\{f\},\{g\}}_{\ell^2(\bbN_0^d)}
=
\jap{b,f g}_{H^2(\bbT^d)}.
$$
Now suppose that there exists $B\in L^\infty(\bbT^d)$ such that the Fourier 
coefficients $B_\kappa$ with  all components $\kappa_j\geq0$ 
coincide with $\han(\kappa)$: 
$$
\jap{B,z^\kappa}_{L^2(\bbT^d)}=\han(\kappa), \quad \kappa\in\bbN_0^d.
$$
In this case we say that $H_d(\han)$ has a \textit{bounded symbol} $B$. 
Since $\jap{b,fg}_{H^2(\bbT^d)}=\jap{B,fg}_{L^2(\bbT^d)}$ we then have that
$$
\abs{\jap{H_d(\han)\{f\},\{g\}}_{\ell^2(\bbN_0^d)}}
\leq 
\norm{B}_{L^\infty(\bbT^d)}\norm{f}_{H^2(\bbT^d)}\norm{g}_{H^2(\bbT^d)}.
$$
This gives the easy ``if'' part of the following theorem.
\begin{theorem}\cite{FergusonLacey,LaceyTerwilleger, Nehari}\label{thm.b1}
The Hankel operator $H_d(\han)$ is bounded on $\ell^2(\bbN_0^d)$ if and only if 
there exists $B\in L^\infty(\bbT^d)$ 
such that
\[
\jap{B,z^\kappa}_{L^2(\bbT^d)}=\han(\kappa), \quad \forall \kappa\in\bbN_0^d.
\label{b1}
\]
Further, there exists a constant $C_d$ such that
\[
\inf\{\norm{B}_{L^\infty(\bbT^d)}\colon  \eqref{b1} \; \mathrm{ holds}\}\leq C_d\norm{H_d(\han)}. 
\label{b2}
\]
When $d=1$, we have $C_1=1$. 
\end{theorem}
When $d=1$, this is the classical Nehari theorem \cite{Nehari}. 
The ``only if'' part for $d=1$ and the bound \eqref{b2} 
follow from the Hahn-Banach theorem and the possibility to factorize any $H^1$-function 
$u$  into a product of two $H^2$-functions: $u=fg$ with $\norm{u}_{H^1}=\norm{f}_{H^2}\norm{g}_{H^2}$. 

For $d=2$ this is a result of Ferguson--Lacey \cite{FergusonLacey}
and for $d>2$ of Lacey--Terwilleger \cite{LaceyTerwilleger}. 
For $d>1$, the strategy of the proof is similar, but instead of the factorization $u=fg$ one has a ``weak factorization" 
$u=\sum_n f_ng_n$ with the estimate 
$$
\sum_n\norm{f_n}_{H^2}\norm{g_n}_{H^2}\leq C_d' \norm{u}_{H^1}.
$$
It is this estimate that requires all the effort in \cite{FergusonLacey,LaceyTerwilleger}.

\subsection{Bounded Helson matrices}
By exactly the same argument as above, we see that the analogue of the easy part of Theorem~\ref{thm.b1}
holds for Helson matrices: if $M(\hel)$ has a \textit{bounded symbol}, that is, if there exists $A\in L^\infty(\bbT^\infty)$ such that 
\begin{equation}
\jap{A,z^\kappa}_{L^2(\bbT^\infty)}=\hel(p^\kappa), \quad \forall \kappa\in\bbN_{0}^{(\infty)},
\label{nehari}
\end{equation}
then $M(\hel)$ is bounded. 
However, not every bounded Helson matrix has a bounded symbol; this was proven 
by Ortega-Cerd\`a and Seip \cite{OS} who observed that for the best constant $C_d$ in \eqref{b2}, 
one has a lower bound
$$
C_d\geq (\pi^2/8)^{d/4}
$$
for all even $d\geq2$. By a standard argument with the closed graph theorem (\cite[Lemma 1]{BP15}),
one obtains
\begin{theorem}\cite{OS}\label{thm.b2}
There exists  a bounded Helson matrix $M(\hel)$ such that there is no $A\in L^\infty(\bbT^\infty)$ 
with the property \eqref{nehari}.
\end{theorem}

To the best of our knowledge, no explicit examples of this kind are available.

\section{Non-negativity and moment sequences} \label{sec:nonneg}

\subsection{Non-negative Hankel matrices}
A $d$-variable ($d<\infty$) Hankel matrix $H_d(\han)$ is non-negative ($H_d(\han)\geq0$), if for any finite
sequence $\{f_\kappa\}_{\kappa\in\bbN_0^d}$  we have
$$
\sum_{\kappa,\kappa'\in\bbN_0^d}\han(\kappa+\kappa')f_\kappa\overline{f_{\kappa'}}\geq0.
$$
For $\tau\in\bbN_0^d$, we will 
denote by $\han(\tau+\cdot)$ the shifted sequence $\{\han(\tau+\kappa)\}_{\kappa\in\bbN_0}$.

For $d=1$, the description of all non-negative Hankel matrices is given by the following 
classical solution to the moment problem: 
\begin{theorem}\label{thm:hankmoment}
Let $\han=\{\han(\kappa)\}_{\kappa=0}^\infty$ be a sequence of complex numbers. 
\begin{enumerate}[label={\rm ({\roman*})}]
\item 
$H_1(\han)$ is non-negative if and only if there exists a measure $\mu\geq0$ on $\bbR$ 
such that $t^\kappa\in L^1(\mu)$ for all $\kappa\in\bbN_0$ and 
\[
\han(\kappa)=\int_\bbR t^\kappa \, d\mu(t), \quad \kappa\in\bbN_0.
\label{hamb}
\]
\item  
We have $H_1(\han)\geq0$ and $H_1(\han(1+\cdot))\geq0$ 
if and only if there exists a measure $\mu\geq0$
as above supported on $[0,\infty)$. 
\end{enumerate}
\end{theorem}
Part (i) is known as the Hamburger moment problem and part (ii) as the Stieltjes moment problem. 
The ``if'' parts are immediate. 
The ``only if'' parts are harder; modern proofs usually rely on von Neumann's theory of self-adjoint extensions 
and on the spectral theorem for self-adjoint operators
(see e.g. \cite[Theorem 7.1]{Peller}). 

It follows immediately from part (ii) that 
$$
H_1(\han)\geq0 \text{ and }H_1(\han(1+\cdot))\geq0
\quad\Rightarrow\quad
H_1(\han(\kappa+\cdot))\geq0\quad \forall\kappa\geq0.
$$ 
This is easy to see directly;  indeed, by considering sub-matrices, 
one checks that $H_1(\han)\geq0$ implies $H_1(\han(2\kappa+\cdot))\geq0$ for
all $\kappa$, and that $H_1(\han(1+\cdot))\geq0$ implies that $H_1(\han(1+2\kappa+\cdot))\geq0$ for all $\kappa$.

The uniqueness of the measure $\mu$ in \eqref{hamb} is a more difficult question. We do not 
discuss it here but refer the reader to \cite{Sim98}, for example. 
For our purposes it suffices to say that 
the measure $\mu$  is unique if the moment sequence $\han(\kappa)$ does not grow extremely fast. 

In the higher-dimensional case $d>1$, the moment problem is considerably more subtle. 
Let $\mu$ be a measure on $\bbR^d$ such that $t^\kappa\in L^1(\mu)$ for all multi-indices $\kappa\in\bbN_0^d$, 
and let  $\han=\{\han(\kappa)\}_{\kappa\in\bbN_0^d}$ be the sequence of moments of $\mu$,
\[
\han(\kappa)=\int_{\bbR^d}t^\kappa \, d\mu(t), \quad \kappa\in\bbN_0^d.
\label{c2}
\]
Then it is evident that $H_d(\han)\geq0$. However, there are sequences $\han$ such that $H_d(\han)\geq0$,
yet $\han$ cannot be represented as a moment sequence of any measure; see \cite[Section 6.3]{BCR}. 
The solvability of the moment problem can be guaranteed under some additional conditions
on the growth rate of the sequence $\han$. 
We give only the simplest theorem of this kind and refer to the extensive literature 
\cite{Ber68,BCR, Infusino, PutinarSchmudgen, Samo91, Vasilescu} for further details. 

\begin{theorem}
Let $d\in\bbN$ and let  $\han=\{\han(\kappa)\}_{\kappa\in\bbN_0^d}$ be a sequence such that for some $C>0$, we have
\[
\abs{\han(\kappa)}\leq C^{\kappa_1+\cdots+\kappa_d}\kappa_1^{\kappa_1}\cdots \kappa_d^{\kappa_d}, 
\quad \forall \kappa\in\bbN_0^d. 
\label{c3}
\]
\begin{enumerate}[label={\rm ({\roman*})}]
\item
$H_d(\han)\geq0$ if and only if there is a positive measure $\mu$ on $\bbR^d$ such 
that for all multi-indices $\kappa\in\bbN_0^d$, one has $t^\kappa\in L^1(\mu)$ and \eqref{c2} holds;
in this case $\mu$ is unique. 
\item
We have $H_d(\han(\kappa+\cdot))\geq0$ for all $\kappa\in\bbN_0^d$ if and only if $\mu$ is supported
on $[0,\infty)^d$. 
\end{enumerate}
\end{theorem}

\subsection{Non-negative Helson matrices}\label{sec.c2}

Generalizing the above, we will say that a Helson matrix $M(\hel)$ is non-negative ($M(\hel)\geq0$), if for any 
finite sequence $\{f_n\}_{n\geq1}$  we have
$$
\sum_{n,m\geq1}\hel(nm)f_n\overline{f_m}\geq0.
$$
For $\ell\in\bbN$, we will 
denote by $\hel(\ell\ \cdot)$ the ``multiplicatively shifted" sequence $\{\hel(\ell n)\}_{n=1}^\infty$.

Below we consider positive measures 
$\mu$ on $\bbR^\infty=\{t=(t_1,t_2,\dots) \, : \, t_j\in\bbR\}$
equipped with the Borel sigma-algebra of the product topology of $\bbR^\infty$, i.e.
the sigma-algebra generated by Borel cylinder sets
$$
\{t\in\bbR^\infty \, : \, t_j\in B_j, j=1,2,\dots,N\},
$$
where $N$ is finite and $B_j$ are Borel sets in $\bbR$. 

The theorem below is a combination of results existing in the literature. We give some comments below. 
\begin{theorem}\cite{Ber68, BCR, Samo91} \label{thm:moment}
Let $\hel=\{\hel(n)\}_{n=1}^\infty$ be a sequence of complex numbers which satisfies 
the bound $\hel(n)=O(n^a)$, $n\to\infty$, for some finite $a>0$. Then
\begin{enumerate}[label={\rm ({\roman*})}]
\item
$M(\hel)\geq0$  if and only if there exists a measure $\mu\geq0$ on $\bbR^\infty$ 
such that $t^\kappa\in L^1(\mu)$ for all  multi-indices $\kappa\in\bbN_{0}^{(\infty)}$ and 
\[
\hel(n)=\int_{\bbR^\infty} t^{\kappa}\, d\mu(t), \quad n=p^\kappa.
\label{eq:moment}
\]
In this case the measure $\mu$ is unique.
\item  
We have $M(\hel)\geq0$ and 
$M(\hel(p\ \cdot))\geq0$ for all primes $p$ if and only if there exists a measure $\mu\geq0$
as above, supported on $[0,\infty)^\infty$. 
\end{enumerate}
\end{theorem}

In both (i) and (ii), the ``if'' part is immediate. 
For completeness we give a sketch of the ``only if'' part below. 

The condition $\hel(n)=O(n^a)$ is only a simple sufficient
condition to ensure that all finite-variable sections of the moment problem satisfy the growth condition \eqref{c3}.

A curious corollary of (ii) is that if $M(\hel)\geq0$ and $M(\hel(p\ \cdot))\geq0$ for all primes $p$, then 
$M(\hel (\ell\ \cdot))\geq0$ for all integers $\ell\geq1$. 
This should be a purely algebraic fact, but we do not know 
how to prove this directly (without the use of the moment problem).

\begin{proof}[Sketch of proof of Theorem~\ref{thm:moment}]
Part (i): the ``if'' part is clear. Let us prove the ``only if'' statement. 
Assume that $M(\hel)\geq0$. 
Let $\mathscr{P}$ be the vector space of all polynomials in $t\in\bbR^\infty$. 
For any such polynomials $f,g$, given by
$$
f(t)=\sum_{n\in\bbN} f_n t^{\kappa(n)}, \quad
g(t)=\sum_{n\in\bbN} g_n t^{\kappa(n)},
$$
consider the positive semi-definite form 
$$
\jap{f,g}=\sum_{n,m\in\bbN} \hel(nm) f_n \overline{g_m}. 
$$
Let us assume for simplicity that $\jap{f,f}=0$ implies $f=0$ (otherwise one needs
to consider the quotient of $\mathscr{P}$ over the corresponding subspace). 
Then $\langle f, f \rangle$ is an inner product on $\mathscr{P}$, 
and we may complete it to a Hilbert space $\mathcal{H}$.

For each $j\geq1$, consider the densely defined operator $A_j$, 
$\dom A_j = \mathscr{P}$, 
given by
$$
(A_j f)(t) = t_jf(t)=\sum_{p_j | n}f_{n/p_j}t^{\kappa(n)}.
$$
Each operator $A_j$ is symmetric, $\langle A_j f , g \rangle = \langle f, A_j g \rangle$, 
and commutes with every other operator $A_k$.
Under the assumption we have made on the growth of $|\hel(n)|$, 
it is known \cite{Sim98} that for every fixed $n_0 \in \bbN$ and for every prime number $p_j$, the Hamburger moment problem
$$
\hel(n_0 p_j^k) = \int_{-\infty}^{\infty} t^k \, d\nu(t), \quad k = 0, 1, 2 \ldots
$$
is uniquely solvable. With this property of unique solvability in each one-dimensional direction, an argument of Devinatz (\cite{Dev57}, 3.1) proves that each operator $A_j$ is actually essentially self-adjoint. That is, the closure of $A_j$ (which we also denote by $A_j$) is self-adjoint.

Hence, we have a countable family $(A_1, A_2, \ldots)$ of commuting self-adjoint operators to which we can apply the spectral theorem \cite{Samo91}. The family has a resolution of identity $E(t)$ on $\mathbb{R}^\infty$  such that
$$
A_j = \int_{\mathbb{R}^\infty} t_j \, dE(t).
$$ 
Hence we obtain
$$
\hel(n) = \left \langle \prod A_j^{\kappa_j} 1, 1\right\rangle 
= 
\int_{\mathbb{R}^\infty} t^\kappa \, \langle dE(t)1, 1 \rangle, \qquad n = p^\kappa, 
$$
which gives the existence part of the theorem. 

Let $j_1, \ldots, j_d$ be distinct positive integers. Under the growth condition on the coefficients, each $d$-variable Hamburger moment problem
$$
\hel(p_{j_1}^{\kappa_{j_1}}\cdots p_{j_d}^{\kappa_{j_d}}) 
= 
\int_{\mathbb{R}^{d}} t_{j_1}^{\kappa_{j_1}}\cdots t_{j_d}^{\kappa_{j_d}} 
\, d\nu(t), 
$$
is uniquely solvable \cite{Ber68}, $1 \leq d < \infty$. Therefore, each projection $\nu$,
$$
\nu(B_{j_1} \times \cdots B_{j_d}) 
= 
\mu\left(\prod_{k=1}^\infty B_k\right), \qquad B_k 
= 
\bbR \textrm{ if } k \notin \{j_1, \ldots, j_d\},
$$
is unique.  
Hence, $\mu$ is also unique, by the Kolmogorov extension theorem.

Part (ii): 
If $\mu$ is supported on $[0,\infty)^\infty$, then it is clear that  $\{\hel(pnm)\}_{n,m=1}^\infty$ is positive semi-definite for every prime $p$. Conversely, if every such form is positive semi-definite, then each operator $A_j$ in the proof is positive semi-definite. Hence the construction (which is unique) leads to a measure $\mu$ supported on $[0,\infty)^\infty$.
\end{proof}

\section{Bounded non-negative matrices and Carleson measures} \label{sec:bounded}

\subsection{Carleson measures}

A Carleson measure for $H^2(\bbT^d)$ (here $d\in\bbN$ or $d=\infty$) 
is a finite (in general complex-valued) measure on the polydisk $\bbD^d$ such that 
\[
\int_{\bbD^d}\abs{f(z)}^2\, d\abs{\mu}(z)\leq C\norm{f}_{H^2(\bbT^d)}^2 
\label{d1}
\]
holds true for some $C>0$ and for all polynomials  $f$. 

For $d=1$, 
a characterisation of Carleson measures in terms of Carleson windows
is well known \cite{Carleson}. 
For $d>1$, it is insufficient to test, for example, on products of Carleson windows, and a sufficient
testing condition incorporates the ``generalized Carleson windows'' of all connected open sets in $\bbT^d$, see \cite{Chang}. However, here we are primarily concerned with measures on $(-1,1)^d$. The intersection between $(-1,1)^d$ and a generalized Carleson window is easy to handle; in this case testing on products of Carleson windows is sufficient, see item (3) of Theorem~\ref{thm.e2}.

Carleson measures $\mu$ for $H^2(\bbT^\infty)$ are poorly understood except in the special case that $\mu$ corresponds, under the Bohr lift (see Section~\ref{sec:bohr}), to a measure of compact support in the half plane $\Re s \geq 1/2$ \cite{Olsen}. See also \cite{BPSVolterra}, where Carleson measures on $H^2(\bbT^\infty)$ are used to characterize the boundedness of Volterra operators.

Testing on reproducing kernels gives a necessary ``Carleson window condition'', which we of course know is insufficient, as this is already the case for $d < \infty$.  In the Appendix we show that the testing condition, analogous to item (3) of Theorem~\ref{thm.e2}, is insufficient even for measures on $(-1,1)^\infty$. 

\subsection{Bounded non-negative Hankel matrices: $d=1$}

For a finite measure $\mu$ in $\bbD^d$, consider the sequence of moments, 
\[
\han(\kappa)=\int_{\bbD^d} z^\kappa \, d\mu(z), \quad \kappa\in\bbN_0^d.
\label{eq:mmt}
\]
Let us begin by discussing $d=1$.
\begin{theorem}\cite[Theorem 7.4]{Peller} \label{thm:balayage}
\begin{enumerate}[label={\rm ({\roman*})}]
\item
If $\mu$ is a Carleson measure and $\han$ is the sequence of moments of $\mu$, 
defined by \eqref{eq:mmt}, then the Hankel 
matrix $H_1(\han)$ is bounded on $\ell^2(\bbN_0)$. 
\item
If $H_1(\han)$ is bounded on $\ell^2(\bbN_0)$, then there exists a Carleson measure on $\bbD$ 
such that $\han$ is given by \eqref{eq:mmt}.  
\end{enumerate}
\end{theorem}

The first part is straightforward, while the second part requires use of the subtle fact that every function of bounded mean oscillation can be represented as the Poisson balayage of a Carleson measure on $\bbD$.

If $H_1(\han)$ is non-negative and bounded, then it is easy to see that the moment measure $\mu$ of $\han$, 
see \eqref{hamb}, must be supported on $(-1,1) \subset \bbD$. 

\begin{theorem}[Widom \cite{Widom66}] \label{thm.e1}
Let $\mu$ be a finite positive measure on $(-1,1)$ and let
$$
\han(\kappa)=\int_{-1}^1 t^\kappa \, d\mu(t), \quad \kappa\in\bbN_0.
$$
Then the following are equivalent:
\begin{enumerate}[label={\rm ({\arabic*})}]
\item
$H_1(\han)$ defines a  bounded operator on $\ell^2(\bbN_0)$.
\item
$\mu$ is a Carleson measure for $H^2(\bbT)$. 
\item 
$\mu$ satisfies the ``Carleson window condition'': 
$\mu(-1,-1+\eps)+\mu(1-\eps,1)=O(\eps)$ as $\eps\to0$. 
\item
The linear operator 
$$
G_1: L^2(\mu)\to L^2(\mu), \quad
(G_1f)(t)=\int_{-1}^1 \frac{f(s)}{1-ts}\, d\mu(s), \quad t\in(-1,1),
$$
is bounded on $L^2(\mu)$.
\item
$\han(n)=O(n^{-1})$ as $n\to\infty$. 
\end{enumerate}
\end{theorem}

We have separated items (2) and (3) in the theorem above only to make it easier to compare it 
with the results that follow.

\subsection{Bounded non-negative Hankel matrices: $d>1$}

For $d > 1$, item (i) of Theorem~\ref{thm:balayage} clearly remains true. 
We do not know if item (ii) is true in this case. 
For measures on $(-1,1)^d$, generating non-negative Hankel matrices, we can say more.
\begin{theorem}\label{thm.e2}
Let $\mu$ be a finite positive  measure on $(-1,1)^d$, $d < \infty$, and let 
$$
\han(\kappa) = \int_{(-1,1)^d} t^\kappa \, d\mu(t), \qquad \kappa = (\kappa_1, \ldots, \kappa_d)\in\bbN_0^d.
$$
Then the following are equivalent:
\begin{enumerate}[label={\rm ({\arabic*})}]
\item 
$H_d(\han)$ defines a bounded operator on $\ell^2(\bbN_0^d)$.
\item
$\mu$ is a Carleson measure for $H^2(\bbT^d)$. 
\item 
There is $C > 0$ such that for every $s \in (-1,1)^d$ it holds that
$$
\mu(I_s) \leq C\prod_{j=1}^d (1-s_j),
$$
where
$$
I_s = \{t \in (-1,1)^d \, ; \, \forall j \geq 1 :  |t_j| \geq |s_j| \; \mathrm{and} \; t_j s_j \geq 0\}.
$$
\item
The linear operator
$$
G_d\colon L^2(\mu)\to L^2(\mu), \quad 
(G_df)(t)= \int_{(-1,1)^d}K_d(s,t)f(s) \, d\mu(s), \quad t\in (-1,1)^d,
$$
is bounded on $L^2(\mu)$, where $K_d$ is the kernel defined in \eqref{eq:kdkernel}.
\item 
There is $C > 0$ such that
\[ 
\label{eq:tensorhilbert}
|\han(\kappa)| \leq C\prod_{j=1}^d \frac{1}{1 + \kappa_j}.
\]
\end{enumerate}
\end{theorem}
\begin{proof}
We will prove 
(2)$\Rightarrow$(3)$\Rightarrow$(5)$\Rightarrow$(1)$\Rightarrow$(4)$\Rightarrow$(2). 

(2)$\Rightarrow$(3): 
Fix $s\in(-1,1)^d$ and apply the defining inequality \eqref{d1} for Carleson measures
to the function $f(z)=K_d(z,s)$. We have
$$
\int_{(-1,1)^d}\abs{K_d(t,s)}^2 \, d\mu(t)
\leq 
C\norm{K_d(\cdot,s)}_{H^2(\bbT^d)}^2
=
CK_d(s,s)
=
C\prod_{j=1}^d (1-s_j^2)^{-1}. 
$$
For the left side, we have 
$$
\int_{(-1,1)^d}\abs{K_d(t,s)}^2\, d\mu(t)
\geq
\int_{I_s}\abs{K_d(t,s)}^2\, d\mu(t)
\geq
\prod_{j=1}^d (1-s_j^2)^{-2}\mu(I_s).
$$
Combining these two inequalities, we obtain (3). 

(3)$\Rightarrow$(5): 
We may assume that $\kappa_j \geq 1$ for every $j$, $1 \leq j \leq d$. For if $\kappa_j = 0$ for some $j$, we may instead reduce to the $(d-1)$-dimensional case. Furthermore, by splitting the integral into its $2^d$ hyperoctants, we may assume that $\mu$ is supported on $(0,1)^d$. Integration by parts in each variable then gives us that
$$
\han(\kappa) = \kappa_1 \kappa_2 \cdots \kappa_d \int_{(0,1)^d} t^{(\kappa_1 -1, \kappa_2 - 1, \ldots, \kappa_d - 1)} \mu(I_t) \, dt.
$$
Applying the hypothesis of (3) gives us that
$$
|\han(\kappa)| \leq C\kappa_1 \kappa_2 \cdots \kappa_d \int_{(0,1)^d} t^{(\kappa_1 -1, \kappa_2 - 1, \ldots, \kappa_d - 1)} \prod_{j=1}^d(1-t_j) \, dt = C\prod_{j=1}^d \frac{1}{1 + \kappa_j}, \qquad t = (t_1, \ldots, t_d),
$$
which is what we wanted to prove.

(5)$\Rightarrow$(1): 
Let $\gamma(j)=1/(1+j)$, and so 
$H_1(\gamma) = \{1/(1+j+k)\}_{j,k=0}^\infty$ is the Hilbert matrix.  Recall that $H_1(\gamma)$ is bounded on $\ell^2(\bbN_0)$
with the norm $\pi$. 
Note that the right-handside of \eqref{eq:tensorhilbert} is the coefficient sequence generating 
as its Hankel matrix the $d$-fold tensor product $H_1(\gamma)^{\otimes d}$, which is a bounded operator with norm $\pi^d$.  
By estimating
$$
\abs{\jap{H_d(\han)\{f_\kappa\},\{g_\kappa\}}_{\ell^2(\bbN_0^d)}}
\leq 
C\jap{H_1(\gamma)^{\otimes d}\{\abs{f_\kappa}\},\{\abs{g_\kappa}\}}_{\ell^2(\bbN_0^d)}
\leq 
C\pi^d\norm{\{\abs{f_\kappa}\}}_{\ell^2(\bbN_0^d)}\norm{\{\abs{g_\kappa}\}}_{\ell^2(\bbN_0^d)},
$$
we obtain the boundedness of $H_d(\han)$.

(1)$\Rightarrow$(4): We follow Widom's argument \cite{Widom66}, adapted here to the multi-dimensional case. 
Let us consider the linear operator
$$
\calN_d\colon \ell^2(\bbN_0^d)\to L^2(\mu),
\quad
\{f_\kappa\}_{\ell^2(\bbN_0^d)}
\mapsto 
\sum_{\kappa\in\bbN_0^d}f_\kappa t^\kappa. 
$$
The operator $\calN_d$ is well defined on the set of all finite sequences $\{f_\kappa\}$. 
Observe that $\calN_d$ is nothing but the Carleson embedding operator
$H^2(\bbT^d)\hookrightarrow L^2(\mu)$ written in the standard basis $\{z^\kappa\}$ of $H^2(\bbT^d)$. 
For any finite sequence $\{f_\kappa\}$, we have
$$
\jap{\calN_d\{f_\kappa\},\calN_d\{f_\kappa\}}_{L^2(\mu)}
=
\sum_{\kappa,\kappa'}f_\kappa \overline{f_{\kappa'}}
\int_{(-1,1)^d}t^{\kappa+\kappa'}\, d\mu(t)
=
\jap{H_d(\han)\{f_\kappa\}, \{f_\kappa\}}_{\ell^2(\bbN_0^d)}. 
$$
Since $H_d(\han)$ is bounded, it follows that $\calN_d$ is bounded. 
(In fact, this argument shows that (1)$\Leftrightarrow$(2).)
It follows that the adjoint is also bounded: 
$$
\calN_d^*\colon  L^2(\mu)\to \ell^2(\bbN_0^d), 
\quad
f\mapsto\left\{\int_{(-1,1)^d}f(t)t^\kappa \, d\mu(t)\right\}_{\kappa\in\bbN_0^d}. 
$$
Further, for any bounded function $f$, compactly supported in $(-1,1)^d$, we have
\begin{align*}
\jap{\calN_d^* f, \calN_d^* f}_{\ell^2(\bbN_0^d)}
&=
\sum_{\kappa}
\int_{(-1,1)^d}f(t)t^\kappa \, d\mu(t)
\int_{(-1,1)^d}\overline{f(s)}s^\kappa\,  d\mu(s)
\\
&=
\int_{(-1,1)^d}
\int_{(-1,1)^d}
K_d(t,s) f(t) \overline{f(s)} 
\, d\mu(s)\,  d\mu(t)
=
\jap{G_d f,f}_{L^2(\mu)},
\end{align*}
and so $G_d$ is non-negative and bounded. 

(4)$\Rightarrow$(2): If $G_d$ is bounded, then the same argument as above shows that 
$\calN_d^*$ is bounded, therefore its adjoint $\calN_d$ is bounded, and 
as already observed this is equivalent to the boundedness of the Carleson embedding 
$H^2(\bbT^d)\hookrightarrow L^2(\mu)$. 
\end{proof}

\begin{remark}
In fact, the above proof shows more than what is claimed. 
Indeed, it is a well known fact that for any bounded operator $\calN$ in a Hilbert space, 
the operators 
$$
\calN^*\calN_{(\Ker \calN)^\perp}\quad\text{ and }\quad \calN \calN^*_{(\Ker \calN^*)^\perp}
$$
are unitarily equivalent. The above proof shows that when (1)--(5) hold, 
then 
$$
H_d(\han)=\calN_d^*\calN_d \quad\text{ and }\quad G_d=\calN_d\calN_d^*.
$$
Hence $H_d(\han)$ and $G_d$ are unitarily equivalent
modulo kernels. This plays an important role in \cite{Widom66}.
\end{remark}

\subsection{Bounded non-negative Helson matrices}

For a finite measure on $\bbD^\infty$, consider the moment sequence $\hel$,
\[
\hel(p^\kappa)=\int_{\bbD^\infty} z^\kappa \, d\mu(z), \quad \kappa\in\bbN_0^{(\infty)}.
\label{eq:mmtpoly}
\]
By the Cauchy-Schwarz inequality we still have the easy part of Theorem~\ref{thm:balayage}.
\begin{theorem}\label{thm.d2}
If $\mu$ is a Carleson measure for $H^2(\bbT^\infty)$ and $\hel$ is defined by \eqref{eq:mmtpoly}, 
then the Helson matrix $M(\hel)$ is bounded on $\ell^2(\bbN)$. 
\end{theorem}
\begin{proof}
This follows from the identity
$$
\jap{M(\hel)\{f_n\},\{g_n\}}_{\ell^2(\bbN)}
=
\int_{\bbD^\infty} f(z) g(z)\, d\mu(z), 
$$
where $f$ and $g$ are the polynomials
$$
f(z)=\sum_{n\geq1} f_n z^{\kappa(n)}, \quad 
g(z)=\sum_{m\geq1} \overline{g_m} z^{\kappa(m)},
$$
and from the definition of Carleson measure. 
\end{proof}

We do not know the answer to the following.
\begin{question}
Let $M(\hel)$ be bounded on $\ell^2(\bbN)$. 
Does there exist a Carleson measure $\mu$ for $H^2(\bbT^\infty)$ such that 
$\hel$ is represented as the moment sequence of $\mu$, as in \eqref{eq:mmtpoly}? 
\end{question}

As already mentioned, when $d > 1$ we do not know if such a result holds even in the finite variable Hankel setting.

Suppose now that $M(\hel)$ is non-negative, so that the moment sequence representation \eqref{eq:moment} holds. 
Note that if $M(\hel) \colon  \ell^2(\bbN) \to \ell^2(\bbN)$ is bounded, then $\hel(n) \to 0$ as $n \to \infty$, 
from which it follows that $\mu$ is concentrated on $(-1,1)^\infty$, i.e. that $\mu(\bbR^\infty \setminus (-1,1)^\infty) = 0$.

Extending the definition of $G_d$ of Theorem~\ref{thm.e2}, 
we would like to define the operator
$$
G_\infty: L^2(\mu)\to L^2(\mu), \quad
(G_\infty f)(t) = \int_{(-1,1)^\infty} K_\infty(s,t) f(s) \, d\mu(s).
$$
However, the convergence of the integral here is no longer an obvious issue, even in the case that $f$ is a polynomial. 
(The problem is illustrated by the fact that there exist Carleson measures for $H^2(\bbT^\infty)$ which are not supported in $\ell^2$,  \cite[Theorem 4.11]{HLS}.)
In the statement below, we have chosen to view $G_\infty$ as the limit (in the strong operator topology) of 
the operators corresponding to the finite variable versions of the kernel $K_\infty$. In order to define these, we need some 
preliminaries. 
We extend the reproducing kernels $K_d(s,t)$ to points $s,t \in (-1,1)^\infty$ by projecting onto the first $d$ coordinates:
\begin{equation}
K_d(s,t) = K_d( (s_1, \ldots, s_d), (t_1, \ldots, t_d)), \qquad s,t \in (-1,1)^\infty.
\label{e0}
\end{equation}
Further, we let the operator $G_d$ of Theorem~\ref{thm.e2} act on $L^2(\mu)$ by the formula
$$
G_d f(t) = \int_{(-1,1)^\infty} K_d(s,t) f(s) \, d\mu(s), \qquad f \in L^2(\mu), \quad t \in (-1,1)^\infty.
$$

\begin{theorem} \label{thm:carleson}
Let $\mu\geq0$ be a finite measure on $(-1,1)^\infty$, and let 
\[
\hel(p^\kappa)=\int_{(-1,1)^\infty} t^\kappa \, d\mu(t), \quad \kappa\in\bbN_0^{(\infty)}.
\label{e1}
\]
Then the following are equivalent: 
\begin{enumerate}[label={\rm ({\arabic*})}]
\item
$M(\hel)$ defines a bounded operator on $\ell^2(\bbN)$. 
\item 
$\mu$ is a Carleson measure for $H^2(\mathbb{T}^\infty)$.
\item  
The operators $G_d \colon  L^2(\mu) \to L^2(\mu)$ are bounded 
and converge in the strong operator topology to a limit $G_\infty \colon  L^2(\mu) \to L^2(\mu)$,
$$
G_\infty f = \lim_{d \to \infty} G_d f, \qquad f \in L^2(\mu).
$$
\end{enumerate}
\end{theorem}
\begin{remark}
The statement is also true if in (3), strong convergence is replaced by convergence in the weak operator topology.
\end{remark} 
\begin{proof}
We follow the proof of Theorem~\ref{thm.e2}. 
Consider the operator
$$
\calN_\infty\colon  \ell^2(\bbN)\to L^2(\mu),
\quad
\{f_n\}_{n=1}^\infty \mapsto \sum_\kappa f_{n(\kappa)} t^{\kappa},
$$
which is well-defined on the dense set of finite sequences $\{f_n\}$. 
This is the embedding operator $H^2(\bbT^\infty)\hookrightarrow L^2(\mu)$ 
written in the basis $\{z^\kappa\}$. 
Since $M(\hel) \geq 0$, the condition of boundedness of $M(\hel)$ can be written as
$$
\sum_{n,m\geq1} \hel(nm)f_n\overline{f_m}
\leq 
C\norm{\{f_n\}}_{\ell^2}^2,
$$
which, using the integral expression \eqref{e1} of $\hel(nm)$, is the same as
$$
\jap{\calN_\infty\{f_n\},\calN_\infty \{f_n\}}_{L^2(\mu)}\leq C\norm{\{f_n\}}_{\ell^2(\bbN)}^2.
$$
This proves
$$
\norm{\calN_\infty}<\infty\Leftrightarrow (1) \Leftrightarrow (2). 
$$
Next, let us prove that the boundedness of $\calN_\infty$ is equivalent to (3). 
First assume that $\calN_\infty$ is bounded. 
Let $P_d \colon  \ell^2(\bbN) \to \ell^2(\bbN)$ be the projection onto the subspace of sequences 
which have non-zero coordinates only for indices $n = p^\kappa$, $\kappa = (\kappa_1, \ldots, \kappa_d, 0, 0, \ldots)$, and 
let $\calN_d = \calN_\infty P_d$. 
Since $\calN_\infty$ is bounded, then its adjoint, 
$$
\calN_\infty^*\colon  L^2(\mu)\to \ell^2(\bbN),
\quad
f\mapsto \biggl\{\int_{(-1,1)^\infty}f(t)t^{\kappa(n)} \, d\mu(t)\biggr\}_{n\in\bbN}
$$
is also bounded, and $\calN_d^* = P_d \calN_\infty^*$ converges to $\calN_\infty^*$ in the strong operator topology. In the Appendix, we show that the set of all  functions of the form 
$t \mapsto g(t_1, \ldots, t_{d'})$, $t \in (-1,1)^\infty$, $1 \leq d' < \infty$, $g$ continuous and compactly supported in $(-1,1)^{d'}$, 
 is dense in $L^2(\mu)$. For such functions $f$ it holds by computation that
$$
\calN_d \calN_d^* f (t) = G_d f(t) = \sum_{\kappa \in \bbN_0^d} \left(\int_{(-1,1)^\infty} f(s) s^\kappa \, d\mu(s) \right) t^\kappa, \qquad t \in (-1,1)^\infty,
$$
and therefore $G_d = \calN_d \calN_d^* = \calN_\infty \calN_d^*$.
Hence each operator $G_d$ is bounded and converges strongly to $G_\infty = \calN_\infty \calN_\infty^*$, so we get (3). 

Conversely, assume (3). 
Note that the adjoint operators $\calN_d^*$ always are densely defined (on the dense set of functions described in the Appendix).
Hence, if the operators $G_d$ converge strongly (or weakly), then 
$$
\sup_d \|\calN_d\|^2 = \sup_d \| \calN_d \calN_d^*\| = \sup_d \|G_d\| < \infty
$$
by the uniform boundedness principle. It follows that $\calN_\infty$ is bounded. 
\end{proof}

\begin{remark}
Again, the above proof shows that when (1)--(3) hold, then 
$$
M(\hel)=\calN_\infty^*\calN_\infty \quad\text{ and }\quad G_\infty=\calN_\infty\calN_\infty^*,
$$
and so $M(\hel)$ and $G_\infty$ are unitarily equivalent
modulo kernels. 
\end{remark}

\begin{remark}
In the Helson matrix case, there is no simple coefficient condition to characterize boundedness, 
as in (5) of Theorem~\ref{thm.e1}. It was shown in \cite{BPSSV} that $M(\hel)$ is bounded and non-negative if
$$
\hel(n)=1/(\sqrt{n}\log n), \quad n \geq 2.
$$
Furthermore, by the method of \cite[Theorem 2]{BPSSV}, the operator $M(\hel')$ is unbounded if $\hel'$ is a sequence such that
$$
\lim_{n\to \infty} \left( \hel'(n) \sqrt{n} \log n\right) = \infty.
$$
On the other hand, let 
$$
\alpha(n)=(1+\kappa_1)^{-1}, \quad\text{ if } n=2^{\kappa_1},\quad \kappa_1\geq0,
$$
and $\alpha(n)=0$ otherwise. 
Then $M(\alpha)$ can be represented as an infinite tensor product of Hankel matrices,
$$
M(\hel) = H_1(\gamma) \otimes H_1(\delta) \otimes H_1(\delta) \otimes \cdots,
$$
where $\gamma(j)=(1+j)^{-1}$ and $\delta$ is the Kronecker's symbol: $\delta(0)=1$ and 
$\delta(j)=0$ for $j > 0$. Since 
$H_1(\gamma)=\{(1+j+k)^{-1}\}_{j,k=0}^\infty$ is the usual Hilbert matrix we see that 
$M(\hel)$ is bounded, while the sequence $\hel$ satisfies
$$
\hel(n) = \frac{1}{1+\log_2 n} \gg \frac{1}{\sqrt{n}\log n}, \quad\text{ if } n=2^{\kappa_1},\quad \kappa_1\geq0.
$$
\end{remark}

\begin{remark}
If $M(\hel)$ is as in Theorem~\ref{thm:carleson}, then 
$M(\hel)$ is realized by the integral operator $H \colon  H^2(\mathbb{T}^\infty) \to H^2(\mathbb{T}^\infty)$, 
$$
Hf(z) = \int_{(-1,1)^\infty} f(t) K_\infty(z,t) \, d\mu(t).
$$
A direct way to see this is to consider, for polynomials $f(z) = \sum_{n=1}^\infty f_n z^{\kappa(n)}$ and $g(z) = \sum_{n=1}^\infty g_n z^{\kappa(n)}$, the computation
$$
\langle Hf, g \rangle_{H^2(\bbT^\infty)} = \sum_{n,m=1}^\infty f_n \overline{g_m} \int_{(-1,1)^\infty} t^{\kappa(n)}t^{\kappa(m)} \, d\mu(t) =  \sum_{n,m=1}^\infty f_n \overline{g_m} \hel(nm), 
$$
which shows that the matrix of $H$ is $M(\hel)$ in the basis $\{z^{\kappa(n)}\}_{n=1}^\infty$.  In particular, the analytic symbol of $M(\alpha)$ is 
$$G(z) = \int_{(-1,1)^\infty} K_\infty(z,t) \, d\mu(t).$$
If $\mu$ is a Carleson measure for $H^1(\bbT^\infty)$ (the closure of polynomials in $L^1(\bbT^\infty)$), a standard argument with the Hahn-Banach theorem shows, in the terminology of Section~\ref{sec:bddnehari}, that $M(\hel)$ has a bounded symbol. In general there are Carleson measures for $H^2(\bbT^\infty)$ which are not Carleson measures for $H^1(\bbT^\infty)$, but we do not know if these can be chosen to be concentrated on $(-1,1)^\infty$. See \cite[pp. 430--431]{BPSSV} for a related question.
\end{remark}

\section{The Bohr lift, Helson matrices, and Hankel integral operators} \label{sec:bohr}

\subsection{Helson matrices generated by measures on $(1/2,\infty)$}

The Bohr lift is the identification between
$$
f(s)=\sum_{n=1}^\infty f_n n^{-s}\quad \text{ and }\quad
F(z)=\sum_{n=1}^\infty f_n z^{\kappa(n)},
$$
where $f$ is a Dirichlet series in $s\in \bbC_0=\{s\in\bbC: \Re s>0\}$
and $F$ is a power series in $z\in\bbD^\infty$. 
Note that the Bohr lift is a ring isomorphism, i.e. it respects multiplication as well as addition. 
In order to write down the Bohr lift more formally, let us define the map
$$
\scrB: \bbC_0\to\bbD^\infty,
\quad
\scrB(s)=(p_1^{-s},p_2^{-s},p_3^{-s},\dots). 
$$
Then the Bohr lift maps $F(z)$ to $f(s)=F(\scrB(s))$. 

The Bohr lift  gives rise to the following 
natural way to generate non-negative Helson matrices.  
Consider a finite positive measure $\nu$ on $(0,\infty)$ and let
\[
\hel(n)=\int_0^\infty n^{-s}\, d\nu(s), \quad n\geq 1.
\label{a0}
\]
Any sequence of the form \eqref{a0} can be written as a moment sequence \eqref{eq:moment}
with the measure $\mu$ being the push-forward of $\nu$ by $\scrB$, i.e.
\[
\hel(n) = \int_{(0,1)^\infty} t^{\kappa(n)} \, d\mu(t), 
\qquad 
\mu(E)=\nu(\scrB^{-1}(E)), \quad E \subset (0,1)^\infty.
\label{eq:pushfor}
\]
It is easy to see that $\nu$ must be concentrated on $(1/2, \infty)$ if $\hel\in \ell^2(\bbN)$. 
Hence, when discussing the boundedness of $M(\hel)$ we can assume this condition. 
The arguments of Section~\ref{sec:bounded} yield a characterization
of bounded non-negative Helson matrices obtained in this way. 

\begin{theorem}
Let $\nu$ be a finite positive measure on $(1/2,\infty)$, and let 
\[
\hel(n)=\int_{1/2}^\infty n^{-s}\, d\nu(s), \quad n\geq 1.
\label{f1}
\]
Then 
the following are equivalent:
\begin{enumerate}[label={\rm ({\arabic*})}]
\item 
$M(\hel)$ is a bounded operator on $\ell^2(\bbN)$.
\item 
There is $C>0$ such that 
$$
\nu((1/2, s]) \leq C(s-1/2), \qquad s \to 1/2.
$$
\item 
There is $C > 0$ such that 
$$
\hel(n) \leq \frac{C}{\sqrt{n}\log n}, \qquad n \geq 2.
$$
\item
The Hankel type integral operator $G\colon L^2(\nu)\to L^2(\nu)$ given by 
$$
G\colon L^2(\nu)\to L^2(\nu), 
\quad 
(Gf)(t)=\int_{1/2}^\infty \zeta(s+t)f(s)\, d\nu(s), \quad t>1/2,
$$
is bounded on $L^2(\nu)$. Here $\zeta(t) = \sum_{n=1}^\infty n^{-t}$ is the Riemann zeta function.
\end{enumerate}
\end{theorem}
\begin{proof}
First let us prove the equivalence of (2) and (3); this is an elementary calculation. 
If (2) holds, then integrating by parts we obtain for $n \geq 2$ that
$$
\hel(n) = \log n \int_{1/2}^\infty n^{-s} \nu((1/2,s]) \, ds \leq C \log n \int_{1/2}^\infty n^{-s} (s-1/2) \, ds = \frac{C}{\sqrt{n}\log n}.
$$
Conversely, assume that (3) holds.  
For any $\delta>0$ we have
$$
\int_{1/2}^\infty n^{-s}\, d\nu(s)
\geq 
\int_{1/2}^{1/2+\delta}n^{-s}\, d\nu(s)
\geq
n^{-\frac12-\delta}\nu((\tfrac12,\tfrac12+\delta]),
$$
and so, using the bound in (3), we get
$$
\nu((\tfrac12,\tfrac12+\delta])\leq Cn^{\delta}(\log n)^{-1}, \quad n\geq2. 
$$
Choosing $n=[e^{1/\delta}]$, we obtain (2). 

Next, let us prove the implication (1)$\Rightarrow$(2).
It can be deduced from Theorem~\ref{thm:carleson} and from the
the necessary condition for a measure on $(-1,1)^\infty$ to be Carleson given by Proposition~\ref{prop:carlcond}. In fact, the measure $\mu$ of \eqref{eq:pushfor} is Carleson for $H^2(\bbT^\infty)$ if and only if (2) holds. However, it is also easy to give a direct argument.
Assume (1) and let $a_n=n^{-\frac12-\delta}$ with $\delta>0$; then 
$$
\sum_{n,m\geq1}\hel(nm)n^{-\frac12-\delta}m^{-\frac12-\delta}
\leq
C
\sum_{n\geq1}n^{-1-2\delta}=C\zeta(1+2\delta).
$$
Writing $\hel$ as a moment sequence, we also have that
\begin{align*}
\sum_{n,m\geq1}\hel(nm)n^{-\frac12-\delta}m^{-\frac12-\delta}
&=
\int_{1/2}^\infty \zeta(1/2+\delta+t)^2\,  d\nu(t)
\\
&\geq
\int_{1/2}^{1/2+\delta} \zeta(1/2+\delta+t)^2\,  d\nu(t)
\geq 
\zeta(1+2\delta)^2 \nu((\tfrac12,\tfrac12+\delta]). 
\end{align*}
Combining these inequalities, we get that
$$
\nu((\tfrac12,\tfrac12+\delta])\leq C/\zeta(1+2\delta)\leq C'\delta,
$$
as $\zeta(t)$ has a first order pole at $t=1$. 

The implication (3)$\Rightarrow$(1) follows from the 
boundedness of the multiplicative Hilbert matrix \eqref{multH} 
(which is a result of \cite{BPSSV}), and from the fact that the sequence $\{n^{-1/2}(\log n)^{-1}\}_{n=2}^\infty$
is in $\ell^2(\bbN_2)$.

All of the above proves the equivalences (1)$\Leftrightarrow$(2)$\Leftrightarrow$(3).
Finally, the equivalence (1)$\Leftrightarrow$(4) is established by the operator-theoretic argument used in Theorems~\ref{thm.e2} and \ref{thm:carleson}. This time, the operator $\calN$ is given by
$$
\calN\colon  \ell^2(\bbN)\to L^2(\nu), 
\quad 
\{a_n\}\mapsto \sum_{n\geq1} a_n n^{-t}, \quad t>1/2,
$$
with the adjoint
$$
\calN^*\colon  L^2(\nu)\to \ell^2(\bbN), 
\quad
f\mapsto \left\{\int_0^\infty n^{-t}f(t) \, d\nu(t)\right\}_{n\in\bbN}. 
$$
As before, $M(\hel) = \calN^* \calN$ is bounded if and only if $G = \calN \calN^*$ is bounded.
\end{proof}

While the equivalences (1)$\Leftrightarrow$(2)$\Leftrightarrow$(3) have appeared implicitly in previous work, 
we point out that the final part of the proof shows that the Helson matrix $M(\hel)$ and the Hankel integral operator $G$ are unitarily equivalent modulo kernels. This observation could have simplified the analysis of the spectrum of the  multiplicative Hilbert matrix which was carried out in \cite{BPSSV}.

\subsection{The multiplicative Hilbert matrix}

The multiplicative Hilbert matrix $\mathbf{H}$, given by \eqref{multH},
was introduced in \cite{BPSSV} as an analogue of the Hilbert matrix $\{1/(1+j+k)\}_{j,k\geq0}$.
The operator $\mathbf{H}$ corresponds to the choice $d\nu(s)=ds$ in \eqref{f1} (with a minor modification to the framework to eliminate $n=1$). 
In the present article's companion paper \cite{PP2}, we will make use of the unitary equivalence 
of $\mathbf{H}$ and a Hankel integral operator $G$ to sharpen the spectral analysis of \cite{BPSSV}. 

\begin{theorem}\cite{BPSSV,PP2}
The multiplicative Hilbert matrix $\mathbf{H}$ is bounded on $\ell^2(\bbN_2)$. 
Its spectrum is purely absolutely continuous (no eigenvalues, no singular continuous component),
has multiplicity one and coincides with $[0,\pi]$. 
\end{theorem}

Despite the formal similarity with the Hilbert matrix, which was diagonalized by M.~Rosenblum,
we do not know how to explicitly diagonalize $\mathbf H$. 
The analysis of \cite{BPSSV,PP2} is based on perturbation theory and
on comparison with  certain diagonalizable integral Hankel operators on $L^2(1/2,\infty)$.

\section{Finite rank Helson matrices} \label{sec:finiterank} 

\subsection{Hankel and Helson forms}
In this section it will be convenient to deal with bilinear (rather than sesquilinear, as in the rest of the paper) 
forms corresponding to Hankel and Helson matrices. We give the appropriate definitions for the case of Helson matrices. The corresponding notions for Hankel matrices are obtained by simply restricting the number of variables.

Suppose as before that $\hel$ is a sequence, generating the Helson matrix $M(\hel)$. 
Let $z = (z_1, z_2, \ldots)$ and let $\bbC[z]$ be the ring of polynomials in the variables $z_1, z_2, \ldots$
For $f, g \in \mathbb{C}[z]$, $f(z) = \sum f_n z^{\kappa(n)}$, $g(z) = \sum g_n z^{\kappa(n)}$, 
define the Helson form $[\cdot, \cdot]_\hel$ by 
$$
[f, g]_\hel = \sum_{n,m \geq 1} f_n g_m \hel(nm).
$$
It is clear that the form $[\cdot, \cdot]_\hel$ satisfies the property
\begin{equation}
[f,g]_\hel = [fg,1]_\hel, \quad f,g \in \bbC[z].
\label{eq:f11}
\end{equation}
Conversely, if $[\cdot, \cdot] \colon  \bbC[z] \times \bbC[z] \to \bbC$ is a bilinear form 
satisfying the property 
$$
[f,g] = [fg,1], \quad f,g \in \bbC[z],
$$ 
then, expanding $f$ and $g$, it is easy to see that this form can be written as 
$[f,g] = [f,g]_\hel$ for the sequence $\hel$ given by $\hel(n) = [z^n, 1]$, $n\geq 1$.
 
One of the aims of introducing the above definition is to enable one to interpret the finite rank 
property of $M(\hel)$ as broadly as possible, and in particular, to separate it from the boundedness of $M(\hel)$. 
We will say that  
$M(\hel)$ (and the form $[\cdot, \cdot]_\hel$) \emph{has finite rank} if the subspace 
$$
\ker [\cdot, \cdot]_\hel = \{f \in \mathbb{C}[z] \, : \, [f,g] = 0, \; \forall g \in \mathbb{C}[z]\}
$$
has finite codimension in $\bbC[z]$.
If the matrix $M(\hel)$ maps finitely supported sequences into $\ell^2(\bbN)$, this corresponds to 
the usual notion of finite rank of the operator $M(\hel) \colon  \ell^2(\bbN) \to \ell^2(\bbN)$. 

A crucial observation is that the property \eqref{eq:f11} implies that  $\ker [\cdot, \cdot]_\hel$ is not only a
vector subspace of $\bbC[z]$, but also an ideal. 
Thus, understanding the algebraic structure of finite rank Hankel forms 
is really a problem in commutative algebra, namely that of describing finite-codimensional ideals in $\bbC[z]$. 
This was observed by Power \cite{Pow82} in the case of finitely many variables.

\subsection{Finite rank Hankel matrices}
Suppose that $\{\han(j)\}_{j=0}^\infty$ is a sequence with the corresponding classical Hankel form
$$
[f,g]_\han = \sum_{j,k \geq 0} f_j g_k \han(j+k),
$$
where $f(w) = \sum_{j} f_j w^j$ and $g(w) = \sum_{j} g_j w^j$ are polynomials 
in the single complex variable $w$. 
Kronecker's theorem on Hankel matrices is usually stated by saying that $[\cdot, \cdot]_\han$ (and equivalently, $H_1(\han)$) has finite rank if and only if the analytic symbol $B$ of $[\cdot, \cdot]_\han$ is rational, where
$$
B(w) = \sum_{j=0}^\infty \overline{\han(j)} w^j. 
$$
Here we have changed the definition of the analytic symbol slightly to fit the bilinearity of $[\cdot, \cdot]_\han$. 
In the notation of Section~\ref{sec:finitenehari} we have that $B(w) = \overline{{b}(\overline{w})}$. 
Note that if all the poles of $B$ lie outside the closed unit disk, then 
$$
\qquad [f, g]_\han = \langle fg, B \rangle_{H^2(\bbT)}, \qquad f,g \in \bbC[w],
$$
and in this case $H_1(\han) \colon  \ell^2(\bbN_0) \to \ell^2(\bbN_0)$ is a bounded operator.

Since we are aiming to work in infinitely many variables, we want to avoid relying on the notion of rational functions. By decomposing $B$ into ``partial fractions'' Kronecker's theorem may be stated as follows.
\begin{theorem}[Kronecker]\label{thm.kronecker}
Suppose that $\lambda_1, \ldots, \lambda_N \in \mathbb{C}$ are a finite number of distinct points, and associate to each $\lambda_j$ a number of constants $c_j(0), c_j(1), \ldots, c_j(k_j)$, where $k_j < \infty$ and $c_j(k_j) \neq 0$. Let $[\cdot, \cdot]$ be the classical Hankel form
$$[f,g] = \sum_{j=1}^N \sum_{k=0}^{k_j} c_j(k) \partial_w^k (fg) (\lambda_j), \qquad f, g \in \bbC[w].$$
Then $[\cdot, \cdot]$ is of finite rank $\sum_{j=1}^N (k_j+1)$ and its analytic symbol is the rational function 
$$
B(w) = \sum_{j=0}^\infty \overline{\han(j)} w^j 
= 
\sum_{j=1}^N \sum_{k=0}^{k_j} \overline{c_j(k)} \partial_{\overline{\lambda}}^k K_1 (w, \lambda_j), \qquad |w| < \min_{1\leq j \leq N} \frac{1}{|\lambda_j|},
$$
where $\han(j) = [w^j, 1]$ and $K_1(w,\lambda) = \frac{1}{1-w\overline{\lambda}}$.
Conversely, any finite rank classical Hankel form can be written in this way. 

The form $[\cdot,\cdot]$ generates a bounded operator $H_1(\han) \colon  \ell^2(\bbN_0) \to \ell^2(\bbN_0)$, 
$\han=\{\han(j)\}_{j=0}^\infty$, if and only if $\lambda_j \in \mathbb{D}$ for $1 \leq j \leq N$. 
In this case, $K_1(w,\lambda)$ can be understood as the reproducing kernel of $H^2(\bbT)$.
\end{theorem}

Except for the rank formula, Power \cite{Pow82} gave a direct generalization of Kronecker's theorem to Hankel forms on $\bbC[z_1, \ldots, z_d] \times \bbC[z_1, \ldots, z_d]$, $d < \infty$. 
Below we use the multi-index notation for differential operators in the standard way, i.e.
$$
\partial_z^\sigma = \partial_{z_1}^{\sigma_1}\cdots \partial_{z_d}^{\sigma_d}, 
\quad
\sigma=(\sigma_1,\dots,\sigma_d)\in\bbN_0^d.
$$
\begin{theorem}\cite{Pow82}\label{thm.power}
Let $1 \leq d < \infty$, and let $[\cdot,\cdot]$ be a Hankel form on 
$\bbC[z_1, \ldots, z_d] \times \bbC[z_1, \ldots, z_d]$. 
This form has finite rank if and only if it can be represented as a finite sum
$$
[f,g] = \sum_{\ell} 
c^{(\ell)}\partial_z^{\sigma^{(\ell)}} (fg) (\lambda^{(\ell)}), 
\qquad f, g \in \bbC[z_1, \ldots, z_d],
$$
where 
$\lambda^{(\ell)}\in\mathbb{C}^d$,  $c^{(\ell)}\in \bbC$, and $\sigma^{(\ell)}\in \bbN_0^d$.  
\end{theorem}

Clearly, a finite rank Hankel form $[\cdot, \cdot]$ on $d$ variables generates a bounded Hankel matrix $H_d(\han) \colon  \ell^2(\bbN_0^d) \to \ell^2(\bbN_0^d)$, $\han(j) = [z^j, 1]$, $j \in \bbN_0^d$, if and only if each $\lambda^{(\ell)} \in \bbD^d$ (assuming that the terms associated with the point $\lambda^{(\ell)}$ do not cancel each other out). Its rational symbol can be obtained as in Kronecker's theorem. We defer the precise statement to the discussion of Helson matrices.

Finding a closed formula for the precise rank of a Hankel form is a difficult combinatorial problem already when $d=2$, see \cite[p. 243]{Pow82} for an enlightening example. Gu \cite{Gu} gave an approach in which computing the rank of a Hankel form on $d$ variables is reduced to computing the rank of several Hankel forms with matrix valued symbols on $d-1$ variables.

\subsection{Finite rank Helson matrices: preliminaries}
In the case that $M(\hel) \colon  \ell^2(\bbN) \to \ell^2(\bbN)$ is bounded and non-negative, the content of Section~\ref{sec:nonneg} yields a characterization of when $M(\hel)$ is of finite rank. 
For if $M(\hel)$ has finite rank, then the Hilbert space $\mathcal{H}$ which appears in the proof sketch of Theorem~\ref{thm:moment} is finite-dimensional. 
Therefore, the moment measure of $\{\hel(n)\}_{n=1}^\infty$ is finitely supported. 
Let the support of the measure consist of the finitely many points  
$\lambda^{(\ell)}=(\lambda^{(\ell)}_1,\lambda^{(\ell)}_2,\dots)\in\bbR^\infty$. 
Then $\hel$ can be represented as a finite sum
$$
\hel(p^\kappa) = \sum_{\ell} c^{(\ell)} (\lambda^{(\ell)})^{\kappa}
$$
with some positive coefficients $c^{(\ell)}$.
Since $M(\hel)$ is bounded, we see that actually 
$\lambda^{(\ell)} \in (-1, 1)^\infty \cap \ell^2(\bbN)$, 
and the analytic symbol of $M(\hel)$ is the function
$$
A(z) = \sum_{\ell} c^{(\ell)} K_\infty(z, \lambda^{(\ell)}),
$$
where $K_\infty(z, \lambda)$ is the reproducing kernel of $H^2(\bbT^\infty)$ (see Section~\ref{sec:prelimhardy}).
Here, again, we refer to $A$ as the analytic symbol in the bilinear sense, meaning that
$$
[f,g]_\hel = \langle fg, A \rangle_{H^2(\mathbb{T}^\infty)} 
= 
\sum_{\ell} c^{(\ell)} f(\lambda^{(\ell)}) g(\lambda^{(\ell)}), \qquad f, g \in \bbC[z].
$$

The characterization of general finite rank Helson forms 
is rather more involved. Before giving a statement, we discuss several examples.
\begin{example}
As in the classical one variable situation, the evaluation of derivatives can also induce finite rank Helson forms. 
The major distinction between Hankel forms on $\bbC[z_1, \ldots, z_d]$, $1 \leq d < \infty$, and our setting is that Helson forms on $\bbC[z]$ have room for directional derivatives in directions which are not finitely supported. 
For a sequence $\{c(j)\}_{j=1}^\infty$ and a point $\lambda \in \bbC^\infty$, consider the Helson form
$$
[f,g] = \sum_{j=1}^\infty c(j) \partial_{z_j}(fg)(\lambda).
$$
We think of $f \mapsto [f,1]$ as the derivative of $f$ in the direction $\{c(j)\}_{j=1}^\infty$ at the point $\lambda$. This directional derivative defines a finite rank Helson form. Indeed,
$$\ker [\cdot, \cdot] = \{f \in \bbC[z] \, : \, f(\lambda) =  \sum_{j=1}^\infty c(j) \partial_{z_j}f(\lambda) = 0 \},$$
and this ideal has codimension $2$. Later we shall see that $[\cdot, \cdot]$ defines a bounded Hankel form on $H^2(\mathbb{T}^\infty) \times H^2(\mathbb{T}^\infty)$ (i.e. $M(\hel) \colon  \ell^2(\bbN) \to \ell^2(\bbN)$ is bounded) if and only if $\{c(j)\}_{j=1}^\infty \in \ell^2(\bbN)$ and $\lambda \in \bbD^\infty \cap \ell^2(\bbN)$. If this is the case, the analytic symbol of the form is
$$
A(z) = \sum_{j=1}^\infty \overline{c(j)} \partial_{\bar{\lambda}_j} K_\infty(z, \lambda).
$$
\end{example}
\begin{example}
The Dirichlet series point of view \cite{HLS} furnishes natural examples. 
As in Section~\ref{sec:bohr}, let us identify a polynomial $F(z)$, $z\in\bbD^\infty$,  
with the Dirichlet polynomial $f(s)=F(\scrB(s))$. 
Under this identification, $H^2(\mathbb{T}^\infty)$ is transformed into the Hardy space $\mathscr{H}^2$ of Dirichlet series, which is a holomorphic function space on the half-plane $\bbC_{1/2} = \{s \in \bbC \, : \, \Re s > 1/2\}$.
 Let $Df = f'$ be the differentiation operator in the half-plane variable $s$.   In terms of the polynomial notation,
$$Df(z) = -\sum_{j=1}^\infty (\log p_j) z_j 
	\partial_{z_j} f(z).$$
Let $\rho$ be a point in $\mathbb{C}$ (corresponding to the point $\{p_j^{-\rho}\}_{j=1}^\infty \in \bbC^\infty$) and let $n$ be a positive integer.  Then the Helson form
$$[f,g] = D^n(fg)(\rho)$$
has kernel
$$\ker [\cdot, \cdot] = \{f \in \bbC[z] \, : \, D^{k}f(\rho) = 0, \; 0 \leq k \leq n\},$$
of codimension $n+1$. 	The Helson form defined in this way is bounded if and only if $\rho \in \bbC_{1/2}$.

\end{example}
\begin{example} \label{ex:nonfinite}
Not every differentiation operator induces a finite rank form. For coefficients $\{c(j)\}_{j=1}^\infty$ and a point $\lambda \in \bbC^\infty$, consider the Helson form
$$[f,g] = \sum_{j=1}^\infty c(j) \partial_{z_j}^2(fg)(\lambda).$$
Then 
$$
\ker [\cdot, \cdot] = \{f \in \bbC[z] \, : \, 
f(\lambda) =  \sum_{j=1}^\infty c(j) \partial_{z_j}^2f(\lambda) 
= 
c(k)\partial_{z_k}f(\lambda) = 0, \; 1 \leq k < \infty \},
$$
which in general does not have finite codimension.
\end{example}

\subsection{Differential operators on $\bbC^\infty$}

Before stating the main theorem of this section,
we need to introduce some definitions. 
For a sequence $c = \{c(j)\}_{j\in\bbN}$ of complex numbers, we say that
$$
D(c) f = \sum_{j=1}^\infty c(j) \partial_{z_j} f
$$
is a first order differential operator; this is a well-defined linear operation on polynomials $f\in\bbC[z]$. 
Given non-zero sequences $c_1, \ldots, c_M$ we call
$$
\bD(\bc) = D(c_1)\cdots D(c_M), 
$$
a \textit{factorizable} differential operator of order $M$, $\ord(\bD(\bc))=M$, 
with $\bc=(c_1,\dots,c_M)$. 
We also consider multiplication by a constant to be a factorizable differential operator of order $0$. 
It is useful to think of $D({c_m})$ as a directional derivative, 
and we call $c_m$ a direction of differentiation of $\bD({\bc})$. 
For $\bc=(c_1,\dots,c_M)$  as above, we write $\bc\in\ell^2$, if $c_m\in \ell^2(\bbN)$ for each $m=1,\dots,M$.

The operator
$\bD({\bc})$ is uniquely defined by its action on the homogeneous polynomials of order $M$, 
and we may therefore think of $\bD({\bc})$ as a symmetric tensor $\odot_{m=1}^M c_m$.  
This point of view will become important in the proof of the main theorem of this section.

\subsection{Finite rank Helson matrices: main result}

We shall now state the main theorem of this section. 
It characterizes the finite rank Helson forms as those given by a finite sum of finite products of directional derivatives. 
This explains why Example~\ref{ex:nonfinite} fails to give a finite rank Hankel form; 
it contains an infinite number of linearly independent directions $(0, \ldots, 0, c(j), 0, \ldots, 0, \ldots)$.

\begin{theorem} \label{thm:helsonkronecker}
Let $[\cdot,\cdot]$ be a Helson form on $\bbC[z]\times\bbC[z]$. Then 
\begin{enumerate}[label={\rm ({\roman*})}]
\item
The form $[\cdot,\cdot]$ has finite rank 
if and only if it can be represented as a finite sum
\begin{equation}
[f,g]
=
\sum_\ell \bD({\bc^{(\ell)}})(fg)(\lambda^{(\ell)}),
\label{eq:finiterankhankel2}
\end{equation}
where for each $\ell$, $\bD({\bc^{(\ell)}})$ is a factorizable differential operator of
a finite order and $\lambda^{(\ell)}\in \bbC^\infty$. 
\item
The form $[\cdot,\cdot]$ is bounded (i.e. defines a bounded Helson matrix) and of finite rank 
if and only if the representation \eqref{eq:finiterankhankel2} 
can be chosen such that for all $\ell$, one has
$\lambda^{(\ell)} \in \bbD^\infty \cap \ell^2$ and  $\bc^{(\ell)}\in \ell^2$. 
\end{enumerate}
\end{theorem}

The proof is given in the next section. 

We shall actually show a little more than stated. 
We will prove that when $[\cdot, \cdot]$ is bounded and of finite rank, 
then the functional $f \mapsto [f,1]$ is continuous on $H^1(\bbT^\infty)$ 
(the closure of polynomials in $L^1(\bbT^\infty)$). 
By the Hahn-Banach theorem, it follows that $M(\hel)$ has a bounded symbol. 
We could have anticipated this behavior, because it is known that 
any Hilbert-Schmidt Helson matrix has a bounded symbol \cite{Helson06}, 
and a bounded finite rank Helson matrix is surely Hilbert-Schmidt.
 
In order to write down the formula for the analytic symbol of a Helson form, we 
define the adjoint operator to $\bD({\bc})$ by 
$$
D({c})^* 
= 
\sum_{j=1}^\infty \overline{c(j)} \partial_{\overline{z_j}}, 
\quad 
\bD({\bc})^*=D({c_1})^*\cdots D({c_M})^*. 
$$
\begin{corollary}\label{cor:bddsymb}
Let $[\cdot,\cdot]$ be a bounded Helson form of finite rank; suppose it is written  as in
\eqref{eq:finiterankhankel2} 
with $\lambda^{(\ell)} \in \bbD^\infty \cap \ell^2$ and  $\bc^{(\ell)}\in \ell^2$ for all $\ell$. 
Then it has the analytic symbol
$$
B(\cdot) = \sum_{\ell}
\bD({\bc^{(\ell)}})^* K_\infty(\cdot, z)|_{z=\lambda^{(\ell)}}.
$$
\end{corollary}
\begin{proof}
This follows from the fact that if $f \in H^2(\mathbb{T}^\infty)$, 
$\lambda \in \bbD^\infty \cap \ell^2$, and $\sigma = (\sigma_1, \sigma_2, \ldots)$ is a finitely supported multi-index, then
\begin{equation*}
\partial_z^\sigma f(z) 
= 
\partial_z^\sigma\jap{f,K_\infty(\cdot,z)}_{H^2(\bbT^\infty)}
=
\jap{f,\partial_{\overline{z}}^\sigma K_\infty(\cdot,z)}_{H^2(\bbT^\infty)}. 
\qedhere
\end{equation*}
\end{proof}

Before embarking on the proof of Theorem~\ref{thm:helsonkronecker}, we also note that given the representation \eqref{eq:finiterankhankel2}, 
it is easy to recognize the  finite rank Helson forms corresponding to non-negative Helson matrices 
(we will call them non-negative finite rank Helson forms).  
They correspond to the case where all factorizable differential operators
$\bD(\bc^{(\ell)})$ are positive constants and all points $\lambda^{(\ell)}$ consist of real coordinates. 
See \cite[p. 242]{Pow82} for the corresponding result in the finite variable case. 

\begin{corollary}
A finite rank Helson form is non-negative if and only if 
there exist finitely many points $\lambda^{(\ell)} \in \bbR^\infty$ and 
finitely many positive numbers $c^{(\ell)}>0$ such that
$$
[f,g] = \sum_{\ell} c^{(\ell)} f(\lambda^{(\ell)}) g(\lambda^{(\ell)}).
$$
\end{corollary}

\section{Proof of Theorem~\ref{thm:helsonkronecker}}\label{sec:frproof}

The proofs of (i) and (ii) of the theorem are based on two distinct constructions, the ``if'' part being easy in both cases.
For (i), the ``only if'' part is based on Power's finite-variable result (Theorem~\ref{thm.power}) and 
on a reduction procedure which depends on elementary polynomial algebra. 
For (ii), the ``only if" part is based on elementary, but rather lengthy, arguments of tensor algebra. 

\subsection{Proof of Theorem~\ref{thm:helsonkronecker}(i): the ``if'' part}
It suffices to prove that the form corresponding to a single term in \eqref{eq:finiterankhankel2}
has finite rank:
$$
[f,g]=\bD({\bc})(fg)(\lambda),
$$
where $\bc=(c_1,\dots,c_M)$. Let 
$$
I_0=\{f\in\bbC[z]: f(\lambda)=0\},
$$
and for every $m=1,\dots,M$, let $I_m$ be the subspace of all polynomials  
$f\in\bbC[z]$ such that $\bD({\mathbf b})f(\lambda)=0$ for all $\mathbf b=(b_1,\dots,b_m)$
such that $\{b_1,\dots b_m\}$ is a subset of $\{c_1,\dots,c_M\}$. 
Now let 
$$
I=\cap_{m=0}^M I_m.
$$
Since every $I_m$ is defined by finitely many linear conditions, $I$ has a finite
codimension. On the other hand, it is easy to see by the product rule that 
$$
I\subset \ker[\cdot,\cdot].
$$
Thus, the Helson form $[\cdot,\cdot]$ has finite rank.
\qed

\subsection{Proof of Theorem~\ref{thm:helsonkronecker}(i): the ``only if'' part}
Below we deal with ideals in rings of polynomials. Any such ideal is necessarily also a linear subspace, 
hence the notions of dimension and codimension are 
unambiguously defined for them.

We need the following description of the ideals of finite codimension in $\bbC[z_1, \ldots, z_d]$, $1 \leq d < \infty$, see \cite[Lemma 3]{Pow82} and \cite[Section 14]{JPR}.
\begin{lemma} \label{lem:finitevarrep}
Let $I$ be a finite-codimensional ideal of $\bbC[z_1, \ldots, z_d]$, 
where $1 \leq d < \infty$. 
Let $\Lambda$ be any linear functional on $\bbC[z_1, \ldots, z_d]$ vanishing on $I$. 
Then $\Lambda$ can be represented as a finite sum 
$$
\Lambda(f)
=
\sum_\ell c^{(\ell)}\partial_z^{\sigma^{(\ell)}}f(z)|_{z=\lambda^{(\ell)}},
$$
where for each $\ell$, $c^{(\ell)}\in\bbC$, $\sigma^{(\ell)}\in\bbN_0^d$ and $\lambda^{(\ell)}\in \bbC^d$. 
\end{lemma}
In the case of ideals $I$ in $\bbC[z]$, $z=(z_1, z_2, \ldots)$, 
we need a more precise statement, allowing for derivatives in directions which are not 
finitely supported. 
The construction in this subsection, 
which gives a characterization of finite-codimensional ideals in $\bbC[z]$,  
was suggested to the authors by Steffen Oppermann.

First we  need a lemma, which allows us to transfer the problem from $\bbC[z]$ 
to $\bbC[z_1,\cdots,z_d]$, for some $d < \infty$.
For $d\in\bbN$, we denote by 
$$
i_d:\bbC[z_1,\dots,z_d]\to \bbC[z], \quad z=(z_1,z_2,\dots),
$$
the natural embedding. 
\begin{lemma}\label{lma.alg}
Let $I\subset \bbC[z]$ be an ideal of finite codimension. 
Then there exists $d\in\bbN$ and a ring homomorphism
$$
\omega:\bbC[z]\to\bbC[z_1,\dots,z_d]
$$
such that $\omega$ is onto, and 
\begin{enumerate}[label={\rm ({\roman*})}]
\item
$\ker \omega\subset I$;
\item
$\omega\circ i_d(g)=g$, $\forall g\in\bbC[z_1,\dots,z_d]$;
\item
$i_d\circ\omega(f)-f\in I$, $\forall f\in\bbC[z]$;
\item
the ideal $\omega(I)$ has finite codimension in $\bbC[z_1,\cdots,z_d]$.
\end{enumerate}
\end{lemma}
\begin{proof}
Let 
$$
\gamma:\bbC[z]\to \bbC[z]/I
$$
be the natural ring-epimorphism.  
We claim that there exists $d \in \bbN$ such that the homomorphism 
$$
\gamma \circ i_d : \bbC[z_1, \ldots, z_d] \to \bbC[z]/I
$$
is onto. To see this, note that 
$$
\gamma \circ i_1( \bbC[z_1]) \subset \gamma \circ i_2( \bbC[z_1, z_2])  \subset \cdots
$$
is an increasing chain of linear subspaces of $\bbC[z]/I$, and
$$
\cup_{d=1}^\infty \gamma \circ i_d( \bbC[z_1, \ldots, z_d] ) = \bbC[z]/I.
$$
Since, by assumption, $\bbC[z]/I$ is finite-dimensional as a vector space, the claim follows.

Now let us define $\omega : \bbC[z] \to \bbC[z_1, \ldots, z_d]$ as follows. 
For each $j \in \bbN$, let $p_j$ be the monomial $p_j(z)=z_j$. 
For $1 \leq j \leq d$, 
set $\omega(p_j)=p_j$, and for $j>d$ let 
$\omega(p_j)$ be any polynomial in $\bbC[z_1,\dots,z_d]$ such that
\begin{equation}
\gamma\circ i_d\circ \omega(p_j)=\gamma(p_j).
\label{eq:f10}
\end{equation}
Now extend $\omega$ to all of $\bbC[z]$ as a ring homomorphism.
The relation \eqref{eq:f10} extends to 
$$
\gamma\circ i_d\circ \omega(f)=\gamma(f), \quad f\in \bbC[z].
$$
From here we obtain property (i) of $\omega$, since $\ker \gamma=I$.   
Property (ii) holds by construction.
In particular, $\omega$ is onto. 
Property (iii) follows immediately from (i) and (ii). 

To verify property (iv), observe that since $\omega$ is onto, we get that $\omega(I)$ is an ideal, and hence the quotient map
$$
\widetilde\omega : \bbC[z]/I \to \bbC[z_1,\cdots, z_d]/\omega(I)
$$
is well-defined and onto. Since the left side here is of finite dimension, so is the right side. 
\end{proof}

\begin{proof}[Proof of Theorem~\ref{thm:helsonkronecker}(i)]
Suppose that a Helson form $[\cdot,\cdot]$ has finite rank.
Let $I=\ker[\cdot,\cdot]$ and let $\omega : \bbC[z] \to \bbC[z_1, \ldots, z_d]$ be as in Lemma~\ref{lma.alg}. 
For $1 \leq j < \infty$, let  $p_j$ be the monomial $p_j(z)=z_j$ and set $P_j = \omega(p_j)$. 
Now let $P: \bbC^d\to\bbC^\infty$ be the map  
$$
P(z_1, \ldots, z_d) = (z_1, \ldots, z_d, P_{d+1}(z_1, \ldots, z_d), \ldots, P_n(z_1, \ldots, z_d), \ldots).
$$
Then the map $\omega$ can be written as 
$$
\omega(f) = f\circ P.
$$
To see this, it suffices to check that the formula holds on monomials $p_j(z)=z_j$, and to note that both sides define ring-homomorphisms.

Now let $\Lambda(f)=[f,1]$. 
By Lemma~\ref{lma.alg}(iii), we have 
$$
\Lambda(f)=\Lambda\circ i_d\circ\omega(f), \quad f\in \bbC[z]. 
$$
By Lemma~\ref{lma.alg}(iv), the ideal $\omega(I)$ in $\bbC[z_1,\dots,z_d]$ has finite 
codimension. Clearly, $\Lambda\circ i_d$ vanishes on this ideal. 
Thus, by Lemma~\ref{lem:finitevarrep}, the linear functional $\Lambda\circ i_d$ on $\bbC[z_1,\dots,z_d]$
has the form 
$$
\Lambda\circ i_d (f)
=
\sum_\ell c^{(\ell)}\partial_{z}^{\sigma^{(\ell)}} f(z)|_{z=\lambda^{(\ell)}},
$$
where the sum is finite. 
To understand the structure of this expression, let us assume first for simplicity, 
that the sum reduces to one term of the first order:
$$
\Lambda\circ i_d(f)=\partial_{z_j}f(z)|_{z=\lambda}.
$$
Then, by the chain rule, 
$$
\partial_{z_j}(f\circ P)=\sum_{k=1}^\infty (\partial_{z_j}P_k)(\partial_{z_k}f)\circ P,
$$
and so 
$$
\Lambda(f)=D({c}) f(z)|_{z=P(\lambda)}, \quad c(k)=\partial_{z_j}P_k(\lambda). 
$$
For higher order terms, we get 
similar but more complicated formulas. However, 
it is clear by the chain and product rules 
that $\Lambda$ can be written on the form \eqref{eq:finiterankhankel2}. 
\end{proof}

\subsection{Proof of Theorem~\ref{thm:helsonkronecker}(ii), the ``if'' part}\label{sec:7.3}
It suffices to consider a single term in  \eqref{eq:finiterankhankel2}, 
$(f,g) \mapsto \bD({\bc})(fg)(\lambda)$, where 
$\bc\in \ell^2$ and $\lambda \in \bbD^\infty \cap \ell^2$.
With $h = fg$, 
note that
$$
\bD({\bc})h(\lambda) 
= 
\partial_{z_1} \cdots \partial_{z_M} h(\lambda + z_1c_1  + \cdots + z_Mc_M)\big|_{(z_1, \ldots, z_M) = 0}.
$$
Any $h \in H^1(\mathbb{T}^\infty)$ is analytic on the subset $\bbD^\infty \cap \ell^2$ 
of the Banach space $\ell^2$ \cite{CG86}. This means precisely that
$$
(z_1, \ldots, z_M) \mapsto h(\lambda + z_1c_1  + \cdots + z_Mc_M)
$$
is analytic in the variables $(z_1, \ldots, z_M)$, in a neighborhood of $0$. 
A routine argument with Cauchy's formula hence shows that
$$
|\bD({\bc})h(\lambda)| \leq C \|h\|_{H^1(\mathbb{T}^\infty)}\leq C\norm{f}_{H^2(\bbT^\infty)}\norm{g}_{H^2(\bbT^\infty)},
$$
demonstrating that $M(\hel)$ is bounded. This argument also proves 
(by the Hahn-Banach theorem) that $M(\hel)$ has a bounded symbol. 
\qed

\subsection{Proof of Theorem~\ref{thm:helsonkronecker}(ii), the ``only if'' part: Step 1} 

The remainder of Section~\ref{sec:frproof} is dedicated to the proof of the ``only if" part of Theorem~\ref{thm:helsonkronecker}(ii). 
The proof is mostly elementary but it contains many steps.
In this subsection, we consider the case when there is only one point $\lambda^{(\ell)}$ in the sum 
\eqref{eq:finiterankhankel2} and all differential operators in this sum are of the same order $m\in\bbN$.

For $m\in\bbN_0$, let $H^2_m(\bbT^\infty)$ be the subspace of $H^2(\bbT^\infty)$ spanned by all homogeneous 
polynomials of degree $m$. 
The main result of this subsection is the following lemma.
\begin{lemma}\label{lma.f15}
Let $m \in \bbN$, $\lambda\in \bbC^\infty$, and let $\Lambda$ be the linear functional on $\bbC[z]$ defined by
\begin{equation}
\Lambda(f)=\sum_k \bD({\bc^{(k)}})f|_{z=\lambda},
\label{eq:f8}
\end{equation}
where the sum is finite and $\ord(\bD({\bc^{(k)}}))=m$ for all $k$. 
Suppose that $\Lambda$ is bounded on $H^2_m(\bbT^\infty)$.  
Then 
$\Lambda$ has a (possibly different) representation of the form \eqref{eq:f8} such that $\bc^{(k)}\in\ell^2$ for every $k$.
\end{lemma}
Note that if $f \in H^2_m(\bbT^\infty)$ and $\ord(\bD({\bc^{(k)}}))=m$, then $f \mapsto \bD({\bc^{(k)}})f|_{z=\lambda}$ is independent of the point $\lambda$. We may therefore consider $f \mapsto \bD({\bc^{(k)}})f$ to be a functional on $H^2_m(\bbT^\infty)$. In the proof of Lemma~\ref{lma.f15} we shall relate this functional to the symmetric tensor generated by $\bc^{(k)}$.

For $m \geq 1$, let $\ell(\bbN^m)$ be the linear space of all sequences $\{c(j)\}_{j\in\bbN^m}$ of complex numbers
(without any metric structure). 
For $c_1,\dots,c_m\in \ell(\bbN)$, the tensor product $c_1\otimes\cdots\otimes c_m$
is defined in the standard way as the element of $\ell(\bbN^m)$ such that 
$$
(c_1\otimes\cdots\otimes c_m)(j)=c_1(j_1)\cdots c_m(j_m), \quad 
j=(j_1,\cdots,j_m).
$$
We denote by $\ell(\bbN)^{\otimes m}\subset \ell(\bbN^m)$ the linear subspace which consists of
all finite linear combinations of such tensor products.

First we need a lemma which is in itself trivial, but involves some unwieldy notation. Therefore, we state it separately:

\begin{lemma}\label{lma.f1}
Let $c^{(1)},\dots,c^{(L)}$ be linearly independent elements of $\ell(\bbN^m)$. 
Then there exist multi-indices $j^{(1)},\dots,j^{(L)}\in\bbN^m$ such that the $L\times L$ matrix
$$
\{c^{(k)}(j^{(\ell)})\}_{\ell,k=1}^L
$$
is invertible. 
\end{lemma}
\begin{proof}
Let us identify $\bbN^m$ with $\bbN$; then the elements $c^{(1)},\dots,c^{(L)}$ are mapped
into sequences of complex numbers labelled by an index in $\bbN$. 
Now the statement of the lemma is immediately recognizable as the elementary theorem saying 
that the row rank of a matrix equals its column rank. 
\end{proof}

\begin{lemma}\label{lma.f2}
Let $c\in \ell^2(\bbN^m)\cap \ell(\bbN)^{\otimes m}$. 
Then $c$ can be represented as a finite linear sum of the form
\begin{equation}
c=\sum_{\ell} c_1^{(\ell)}\otimes\cdots\otimes c_m^{(\ell)},
\label{eq:f1}
\end{equation}
where $c_1^{(\ell)},\dots,c_m^{(\ell)}\in \ell^2(\bbN)$ for all $\ell$. 
\end{lemma}
\begin{proof}
\emph{Step 1:}
To begin, let us prove that there exists a representation \eqref{eq:f1} with the last entry
$c_m^{(\ell)}\in\ell^2(\bbN)$ for all $\ell$. Later we will inductively deal with the other entries.
We start with a representation 
\begin{equation}
c=\sum_{\ell=1}^L a_1^{(\ell)}\otimes\cdots\otimes a_m^{(\ell)},
\label{eq:f2}
\end{equation}
where $a_1^{(\ell)},\dots,a_m^{(\ell)}\in \ell(\bbN)$; such a representation exists
by the assumption that $c\in\ell(\bbN)^{\otimes m}$. 
Consider the tensor products
\begin{equation}
a_1^{(\ell)}\otimes\cdots\otimes a_{m-1}^{(\ell)}, \quad \ell=1,\dots,L.
\label{eq:f3}
\end{equation}
First we want to reduce the problem to the case when the $L$ products of
\eqref{eq:f3} are linearly independent in $\ell(\bbN^{m-1})$. 
Suppose that they are linearly dependent. Then we can express one of them
as a linear combination of the other ones. 
Substituting this expression back into the sum \eqref{eq:f2}, we reduce this sum to a similar 
one with $L-1$ terms. Continuing this way, eventually we will be able to reduce the 
problem to the case when all the tensor products in \eqref{eq:f3} are linearly independent. 

Next, assuming that these reductions have already been made, so that the $L$ terms in 
\eqref{eq:f3} are linearly independent, let us apply Lemma~\ref{lma.f1} to these terms. 
This ensures that we can chose multi-indices
$j^{(1)},\dots,j^{(L)}\in\bbN^{m-1}$ such that the $L\times L$ matrix
$$
\{A_k^{(\ell)}\}_{\ell,k=1}^L,
\quad
A_k^{(\ell)}=a_1^{(\ell)}(j_1^{(k)})\dots a_{m-1}^{(\ell)}(j_{m-1}^{(k)}),
$$
is invertible. Next, our assumption $c\in \ell^2(\bbN^m)$ 
implies, in particular, that
$$
\sum_{j_m=1}^\infty \abs{c(j)}^2<\infty,
\quad
j=(j_1^{(k)},\dots,j_{m-1}^{(k)},j_m),
$$
for each of our chosen multi-indices $j^{(k)}$. Noting that
$$c(j) = \sum_{\ell=1}^L A_k^{(\ell)}a_m^{(\ell)}(j_m), \quad
j=(j_1^{(k)},\dots,j_{m-1}^{(k)},j_m),$$
we hence have that
$$
\sum_{\ell=1}^L A_k^{(\ell)}a_m^{(\ell)}\in\ell^2(\bbN),
\quad
k=1,\dots,L.
$$
Since the matrix $\{A_k^{(\ell)}\}$ is invertible, it follows
that $a_m^{(\ell)}\in \ell^2(\bbN)$ for every $\ell=1,\dots,L$. This finishes the first step of the proof.

\emph{Step 2:}
Now we start from a representation \eqref{eq:f2} already set up so that $a_m^{(\ell)}\in \ell^2(\bbN)$ for $\ell=1,\dots,L$. 
In this step we prove that there exists a possibly different representation \eqref{eq:f2} where
both $a_m^{(\ell)}\in \ell^2(\bbN)$ and $a_{m-1}^{(\ell)}\in \ell^2(\bbN)$ for every $\ell$. 
As in the previous step, we consider the tensor products
$$
a_1^{(\ell)}\otimes\cdots\otimes a_{m-2}^{(\ell)}\otimes a_m^{(\ell)}, \quad \ell=1,\dots,L.
$$
If these tensor products are linearly dependent, 
we eliminate the linear dependence as before, observing that the process of elimination preserves the property that $a_m^{(\ell)} \in \ell^2(\bbN)$ for every $\ell$. Once linear independence is established, we see that $a_{m-1}^{(\ell)}\in \ell^2(\bbN)$,
$\ell=1,\dots,L$, by using Lemma~\ref{lma.f1} exactly as in the previous step.

\emph{Step 3:}
We now repeat the process until we arrive at a representation with all $c_1^{(\ell)},\dots c_m^{(\ell)}\in \ell^2(\bbN)$. 
\end{proof}

To prepare for the proof of Lemma~\ref{lma.f15}, we next consider symmetric tensors. 
For $c_1,\dots,c_m\in \ell(\bbN)$, the symmetrization operator $\Sym$ is defined by
\begin{equation}
\Sym(c_1\otimes\dots\otimes c_m)
= c_1\odot\cdots\odot c_m =
\frac1{m!}\sum_{\sigma}c_{\sigma(1)}\otimes\cdots \otimes c_{\sigma(m)},
\label{eq:f5}
\end{equation}
where the sum is taken over all permutations $\sigma$ of the set $\{1,\dots,m\}$. 
This operator is idempotent, $\Sym(\Sym(c))=\Sym(c)$. 
We denote by $\ell(\bbN)^{\odot m}$ 
 the linear space 
which consists of all finite sums of  
symmetric products of the form \eqref{eq:f5}. Clearly, we have
$$
\ell(\bbN)^{\odot m}\subset \ell(\bbN)^{\otimes m}\subset \ell(\bbN^m).
$$

\begin{lemma}\label{lma.f3}
Let $c\in\ell^2(\bbN^m)\cap\ell(\bbN)^{\odot m}$.
Then $c$ can be represented as a finite sum of the form 
\begin{equation}
c=\sum_\ell c_1^{(\ell)}\odot\cdots\odot c_m^{(\ell)},
\label{eq:f6}
\end{equation}
where $c_1^{(\ell)},\dots,c_m^{(\ell)}\in \ell^2(\bbN)$ for all $\ell$. 
\end{lemma}
\begin{proof}
Since $\ell(\bbN)^{\odot m}\subset \ell(\bbN)^{\otimes m}$, by Lemma~\ref{lma.f2} there exists  a representation of $c$ as a finite sum
\eqref{eq:f1} such that $c_1^{(\ell)},\dots,c_m^{(\ell)}\in \ell^2(\bbN)$. 
Applying $\Sym$ to both sides of \eqref{eq:f1} and observing that $\Sym(c)=c$,
we arrive at \eqref{eq:f6}.
\end{proof}

\begin{proof}[Proof of Lemma~\ref{lma.f15}]
We start with some combinatorial preliminaries. 
The order $m\in\bbN$ is fixed throughout the proof. 
For $j\in\bbN^m$, it will be convenient to use the notation $z_j=\prod_{\ell=1}^m z_{j_\ell}$; 
this should not be confused with the earlier used notation $z^\kappa=\prod_{\ell=1}^\infty z_\ell^{\kappa_\ell}$
for $\kappa\in \bbN_0^{(\infty)}$. 
Observe that for $\kappa\in\bbN_0^{(\infty)}$ with $\abs{\kappa}=m$, we have
$$
\partial_{z_{j_1}}\cdots\partial_{z_{j_m}} z^\kappa=
\begin{cases}
\kappa!:=\kappa_1!\kappa_2!\dots, & \text{ if $z_j=z^\kappa$,}
\\
0& \text{ otherwise.}
\end{cases}
$$
Further, for any $j\in \bbN^m$, the set $\{j': z_{j'}=z_j\}$ coincides with the set of all permutations
of the coordinates of $j$. Hence the number of elements of this set is given by 
by the multinomial coefficient
$$
|\{j: z_j=z^\kappa\}| = \binom{|\kappa|}{\kappa}:=
\frac{\abs{\kappa}!}{\kappa!}. 
$$

With these preliminaries in mind, let 
$\bD({\ba})$ be a single factorizable differential operator of order $m$, $\ba=(a_1,\dots,a_m)$. 
Then, for any multi-index $\kappa$ with $\abs{\kappa}=m$, we have 
$$
\bD({\ba})z^{\kappa}
=
\kappa!\sum_{j: z_j=z^\kappa} (a_1\otimes \cdots \otimes a_m)(j)
=
m! (a_1\odot \cdots \odot a_m)(j_*),
$$
where $j_*$ is any element of the set $\{j: z_j=z^\kappa\}$. 
This can be alternatively rewritten as
\[
\bD({\ba})z^{\kappa}
=
m!\binom{m}{\kappa}^{-1} \sum_{j: z_j=z^\kappa} (a_1\odot \cdots \odot a_m)(j)
=
\kappa! \sum_{j: z_j=z^\kappa} (a_1\odot \cdots \odot a_m)(j);
\label{eq:f7a}
\]
all terms in the last sum are equal to each other. 

Next, let  $\Lambda$ be as in \eqref{eq:f8}, with coefficients $\bc^{(k)}=(c_1^{(k)},\dots,c_m^{(k)})$, and let
$$
\bc=\sum_k c_1^{(k)}\odot\cdots \odot c_m^{(k)}.
$$
Then, by linearity \eqref{eq:f7a} extends to
$$
\Lambda(z^\kappa)=\kappa! \sum_{j: z_j=z^\kappa} \bc(j). 
$$

Further, since the Hilbert space $H^2_m(\bbT^\infty)$ has the orthonormal basis $\{z^\kappa\}_{|\kappa|=m}$, 
the boundedness of $\Lambda$ on $H^2_m(\bbT^\infty)$ means that
$$
\sum_{\kappa: \abs{\kappa}=m}\abs{\Lambda(z^\kappa)}^2<\infty.
$$
This yields 
$$
\sum_{\kappa: \abs{\kappa}=m}(\kappa!)^2 \Abs{\sum_{j: z_j=z^\kappa} \bc(j)}^2<\infty.
$$
Since all terms in the sum over $j$ here are equal to each other, we have
$$
\Abs{\sum_{j: z_j=z^\kappa} \bc(j)}^2=\binom{\abs{\kappa}}{\kappa}\sum_{j: z_j=z^\kappa}\abs{\bc(j)}^2,
$$
and so we obtain 
$$
\sum_{\kappa: \abs{\kappa}=m}\sum_{j: z_j=z^\kappa} \abs{\bc(j)}^2
\leq
\sum_{\kappa: \abs{\kappa}=m}\sum_{j: z_j=z^\kappa} 
(\kappa!)^2\binom{\abs{\kappa}}{\kappa}\abs{\bc(j)}^2
<\infty.
$$
The double sum here is simply the sum over $j\in\bbN^m$. 
Hence, we have shown that $\bc\in \ell^2(\bbN^m)$. 
Applying Lemma~\ref{lma.f3}, we obtain that $\bc$ can be represented in the form \eqref{eq:f6}
with all coefficients in $\ell^2(\bbN)$; 
this means that $\Lambda$ can be represented on
the form \eqref{eq:f8} with all $\bc^{(k)}\in \ell^2(\bbN^m)$.
\end{proof}

\subsection{Proof of Theorem~\ref{thm:helsonkronecker}(ii), the ``only if'' part: Step 2} 
Here we complete the proof. Our main argument is that if we split the sum
\eqref{eq:finiterankhankel2} into separate parts corresponding to different points
$\lambda^{(\ell)}$ and to  differential operators of different orders, 
then each part will give rise to a bounded linear functional to which we can apply Lemma~\ref{lma.f15}.

Before going into details, we record for ease of reference three simple facts that underpin our argument.
\begin{proposition}\label{prp.auxiliary}
The following facts are true.
\begin{enumerate}[label={\rm ({\roman*})}]
\item
A point evaluation functional,
$$
f\mapsto f(\lambda), \quad f\in \bbC[z], \quad \lambda\in \bbC^\infty,
$$
is bounded on $H^2(\bbT^\infty)$ if and only if $\lambda\in \bbD^\infty\cap \ell^2$. 
\item
Let $f\mapsto \Lambda(f)$ be a bounded linear functional on $H^2(\bbT^\infty)$. 
Then for any $p\in\bbC[z]$, the linear functional $f\mapsto \Lambda(p f)$
is also bounded on $H^2(\bbT^\infty)$.
\item
For any factorizable differential operator $\bD(\bc)$ and any
$\lambda\in \bbC^\infty$, one has the ``commutation formula'' 
\begin{equation}
\bD({\bc})(f_1f_2)(\lambda)=f_1(\lambda)\bD({\bc})f_2(\lambda)+\bD({\mathbf b})f_2(\lambda),
\label{eq:commut}
\end{equation}
where $\ord(\bD(\mathbf b))<\ord(\bD(\bc))$, and  $\mathbf b$ depends on $\bc$,  $f_1$, $\lambda$.
\end{enumerate}
\end{proposition}
\begin{proof}
(i) goes back to \cite{CG86} and can also be seen directly from the power series representation
$$
f(\lambda)=\sum_\kappa \jap{f,z^\kappa}_{H^2}\lambda^\kappa.
$$
(ii) is immediate from the estimate $\norm{pf}_{H^2}\leq \norm{p}_{L^\infty}\norm{f}_{H^2}$ and (iii) is a direct calculation. 
\end{proof}

First we consider the case of only one point $\lambda^{(\ell)}$ in the sum 
\eqref{eq:finiterankhankel2}.

\begin{lemma}\label{lma.f13}
Let $\lambda\in\bbC^\infty$ and let $\Lambda$ be a nonzero linear functional on $\bbC[z]$, bounded 
on $H^2(\bbT^\infty)$, of the form
\begin{equation}
\Lambda=\sum_{m=0}^M \Lambda^m, \quad
\Lambda^m(f)=\sum_{k} \bD({\bc^{(m,k)}})f(\lambda),
\label{eq:f13}
\end{equation}
where all sums are finite and $\ord(\bD({\bc^{(m,k)}}))=m$ for every $m$ and $k$.
Then $\lambda\in\bbD^\infty\cap \ell^2$. Furthermore, all terms $\Lambda^m$ are bounded on $H^2(\bbT^\infty)$ and can be represented as in \eqref{eq:f13} with all $\bc^{(m,k)}\in\ell^2(\bbN)$. 
\end{lemma}
\begin{proof}
Without loss of generality, let us assume that the highest order term is non-zero: $\Lambda^M\not=0$. 
First let us prove that $\lambda\in\bbD^\infty\cap\ell^2$. 
If $M=0$, then this follows from Proposition~\ref{prp.auxiliary}(i). 
Suppose instead that $M\geq1$.  
Since $\Lambda^M\neq0$, there is a multi-index $\kappa$,  $\abs{\kappa}=M$, such that $\Lambda(p)\neq0$
for the polynomial $p(z)=(z-\lambda)^\kappa$.
Applying the commutation formula \eqref{eq:commut} with $f_1=f$, $f_2=p$, we obtain that
$$
\Lambda(pf)=\Lambda(p)f(\lambda),
$$
since $\bD({\mathbf b})p(\lambda)=0$ whenever $\ord (\bD({\mathbf b}))<M$. 
By Proposition~\ref{prp.auxiliary}(ii) the functional $f\mapsto \Lambda(pf)$ is bounded on $H^2(\bbT^\infty)$, and hence it follows 
that $\lambda\in\bbD^\infty\cap\ell^2$. 

Next, let us prove that each $\Lambda^m$ is bounded and can be represented as in \eqref{eq:f13}
with all $\bc^{(m,k)}\in\ell^2$. 
We have already seen that $\Lambda^0$ is bounded. 
Assuming that $\Lambda^0,\dots,\Lambda^{m-1}$ are bounded, for any homogeneous
polynomial $f$ of degree $m$ we have
$$
\Lambda^m(f)=\Lambda(f)-\sum_{s=0}^{m-1}\Lambda^s(f).
$$
The right hand side of this expression is bounded on $H^2(\bbT^\infty)$ by assumption. 
Hence $\Lambda^m$ is bounded on the subspace $H^2_m(\bbT^\infty)$. 
By Lemma~\ref{lma.f15}, we have a representation  \eqref{eq:f13} for $\Lambda^m$
with all coefficients $\bc^{(m,k)}\in\ell^2$. 
By the argument of Section~\ref{sec:7.3}, it follows that $\Lambda^m$ is bounded on the whole space $H^2(\bbT^\infty)$. 
Thus, by induction, all $\Lambda^m$ are bounded and can be written on the required form. 
\end{proof}

\begin{proof}[Proof of Theorem~\ref{thm:helsonkronecker}(ii), the ``only if'' part]

Suppose that we have a representation \eqref{eq:finiterankhankel2} such that 
the corresponding form 
$$
\Lambda(f)=[f,1]
$$
is bounded on $H^2(\bbT^\infty)$. 
Let us change our notation slightly and rewrite $\Lambda$ on the form 
\[
\Lambda=\sum_{m=0}^M\sum_\ell \Lambda^m_\ell, 
\quad\text{ with }\quad
\Lambda_\ell^m(f):=\sum_{k} \bD({\bc^{(m,\ell,k)}})f(\lambda^{(\ell)}),
\label{eq:f12}
\]
where the points $\lambda^{(\ell)}$ are \emph{distinct}, 
all sums are finite and $\ord(\bD({\bc^{(m,\ell,k)}}))=m$ for every $m$, $\ell$, and $k$.
Let us prove that every term $\Lambda_\ell^m$ is bounded on $H^2(\bbT^\infty)$.
We will use induction in $M$.

Suppose first that $M=0$, i.e. that $\Lambda$ is a sum of evaluation functionals.
If there is only one point $\lambda^{(\ell)}$, the statement follows from Proposition~\ref{prp.auxiliary}(i).  
Suppose that there is more than one point $\lambda^{(\ell)}$. 
Fix some value of $\ell$, say $\ell=1$. 
For each $\ell\not=1$, choose $j_\ell$ such that $\lambda^{(\ell)}_{j_\ell}\not=\lambda^{(1)}_{j_\ell}$, and let
\begin{equation}
p(z)=\prod_{\ell\not=1}(z_{j_\ell}-\lambda^{(\ell)}_{j_\ell}).
\label{eq:f14}
\end{equation}
Then $p(\lambda^{(1)})\not=0$ and $p(\lambda^{(\ell)})=0$ for all $\ell\not=1$. 
It follows that 
$$
\Lambda(pf)=p(\lambda^{(1)})\Lambda_1^0(f),
$$
and so, by Proposition~\ref{prp.auxiliary}(ii), the functional $\Lambda_1^0$ is bounded. 
Applying this argument to every point $\lambda^{(\ell)}$ we obtain
that $\Lambda^0_\ell$ is bounded for every $\ell$. 

Now assume that the required statement is proven for some $M>0$; we will then prove it for $M$ replaced by $M+1$. 
Let us first prove that $\Lambda_1^{M+1}$ is bounded, assuming that it is non-zero.
Let $p$ be the  polynomial of \eqref{eq:f14} and consider its power $p^{M+2}$. 
Then it is clear that 
$$
\Lambda_\ell^m(p^{M+2}f)=0,\quad \forall \ell\not=1, \quad  m=0,1,\dots,M+1.
$$
It follows that 
$$
\Lambda(p^{M+2}f)=\sum_{m=0}^{M+1}\Lambda_1^m (p^{M+2}f).
$$
By the commutation formula \eqref{eq:commut},
the last functional can be written as 
\begin{equation}
\sum_{m=0}^{M+1}\Lambda_1^m (p^{M+2}f)
=
p^{M+2}(\lambda^{(1)})\Lambda_1^{M+1}(f)
+\wt\Lambda_1(f),
\label{eq:f7}
\end{equation}
where $\ord(\wt\Lambda_1)\leq M$. 
Since the linear functional 
$f\mapsto \Lambda(p^{M+2}f)$ is bounded by Proposition~\ref{prp.auxiliary}(ii),
the linear functional on the right hand side of \eqref{eq:f7} is also bounded. 
Applying Lemma~\ref{lma.f13}, we find that the top order functional $\Lambda_1^{M+1}$ is bounded. 

Repeating the above argument for every $\ell$, we obtain that all the highest order terms $\Lambda_\ell^{M+1}$
are bounded. Subtracting them from $\Lambda$, we obtain that the functional
$$
\sum_{m=0}^M\sum_\ell \Lambda_\ell^{m}
$$
is bounded.
By the inductive hypothesis, it follows that all $\Lambda_\ell^m$ are bounded. 

Now applying Lemma~\ref{lma.f13} to each term $\Lambda_\ell^m$, we get that 
$\lambda^{(\ell)}\in \bbD^\infty\cap\ell^2$ and that  $\Lambda_\ell^m$ can be rewritten on the required form \eqref{eq:f12}
with $\bc^{(m,\ell,k)}\in \ell^2$. This yields the desired representation of $\Lambda$. 
\end{proof}

\appendix
\section{Carleson measures on $(-1,1)^\infty$}

Let $\mu$ be a positive Borel measure on $\bbR^\infty$, as in Section~\ref{sec.c2}. 
 The measures $\mu$ considered in Section~\ref{sec:bounded} are concentrated on $(-1,1)^\infty$, which means that $\mu(\bbR^\infty \setminus (-1,1)^\infty) = 0$. In this case $\mu$ may be considered as a Borel measure on $(-1,1)^\infty$, equipped with the usual product topology. This is so because the relative sigma-algebra on $(-1,1)^\infty$, induced by the Borel sets of $\bbR^\infty$, coincides with the Borel sigma-algebra of $(-1,1)^\infty$. 
\begin{proposition} \label{prop:polydense}
Let $\mu$ be a finite Borel measure on $(-1,1)^\infty$. Then functions of the form $t \mapsto f((t_1, \ldots, t_d))$, $t \in (-1,1)^\infty$, $d < \infty$, $f$ continuous and compactly supported in $(-1,1)^d$, are dense in $L^2(\mu)$. In particular,  polynomials 
$$f(t) = \sum_{\kappa} a_\kappa t^\kappa, \qquad a_\kappa \in \bbC,$$
span a dense set in $L^2(\mu)$.
\end{proposition}
\begin{proof}
Let $\Gamma$ be the \textit{ring} of sets generated by the Borel cylinder sets of $(-1,1)^\infty$. Then the linear span of characteristic functions of sets in $\Gamma$ is dense in $L^2(\mu)$. Any such characteristic function only depends on, say, the first $d < \infty$ variables. The relative (projected) measure $\mu_d$ on $(-1,1)^d$ is a finite Borel measure on the locally compact space $(-1,1)^d$, and is hence a regular measure. Therefore any characteristic function of the above type may be approximated, in $L^2(\mu)$, by functions of the type mentioned in the statement. The last claim now follows by the Stone-Weierstrass theorem.
\end{proof}
Let us now discuss the natural testing condition for Carleson measures $\mu$ for $H^2(\bbT^\infty)$, in the case
that $\mu$ is supported on $(-1,1)^\infty$. 
As in \eqref{e0}, we extend the finite-variable reproducing kernel $K_d(s,t)$ to $s,t\in\bbD^\infty$. 
Fix $s \in (-1,1)^\infty$ and for $d \geq 1$ consider the function $z\mapsto K_d(z,s)$, where $z\in\bbD^\infty$.  
Clearly, this function can be approximated by polynomials in the variables $z_1, \ldots, z_d$, 
with uniform convergence in the entire polydisc $\bbD^\infty$. 
Hence, if $\mu$ is Carleson measure for $H^2(\bbT^\infty)$, the defining 
inequality \eqref{d1} for Carleson measures extends to hold for $f(z)=K_{d}(z,s)$, implying that
\begin{align}
\label{eq:reprkernel}
\int_{(-1,1)^\infty} |K_{d}(t,s)|^2 \, d\mu(t) &\leq C\|K_{d}(\cdot,s)\|^2_{H^2(\bbT^\infty)} \nonumber \\ 
&= CK_d(s,s) = C\prod_{j=1}^d \frac{1}{1-s_j^2}.
\end{align}
This leads to the following proposition. 
\begin{proposition} \label{prop:carlcond}
Suppose that $\mu$ is a measure on $(-1,1)^\infty$ 
which is a Carleson measure for $H^2(\bbT^\infty)$. 
For any $s \in (-1,1)^\infty$, consider the set
$$I_s = \{t \in (-1,1)^\infty \, ; \, \forall j \geq 1 : |t_j| \geq |s_j| \; \mathrm{and} \; t_j s_j \geq 0\}.$$
Then there is $C > 0$ such that the Carleson condition
\[ \label{eq:carlcond} \mu(I_s) \leq  C\prod_{j=1}^\infty (1-s_j^2), \]
holds, where if $s \notin \ell^2(\bbN)$ the right-hand side is to be understood as $0$. 
\end{proposition}
\begin{proof}
Note that when $s \in (-1,1)^\infty$ we have
$$|K_{d}(t,s)|^2 \geq \prod_{j=1}^d (1-s_j^2)^{-2}, \qquad t \in I_s.$$
From \eqref{eq:reprkernel} we deduce that 
$$\mu(I_s) \leq  C\prod_{j=1}^d (1-s_j^2).$$
In the limit as $d \to \infty$ we obtain \eqref{eq:carlcond}. 
\end{proof}
In the classical context of the Hardy space $H^2(\bbT)$, testing on reproducing kernels is sufficient to guarantee that a measure is Carleson. In our relatively simple situation of measures on $(-1,1)^\infty$, one could conjecture that $\mu$ is Carleson for $H^2(\bbT^\infty)$ if \eqref{eq:reprkernel} holds. However, this turns out to be false.
\begin{proposition} \label{prop:notcarleson}
There is a positive measure $\mu$, concentrated on $(-1,1)^\infty$, which satisfies
$$
\sup_{\substack{s \in (-1,1)^\infty \\ d\geq1}} \left( \prod_{j=1}^d (1-s_j^2) \int_{(-1,1)^\infty} |K_{d}(t,s)|^2 \, d\mu(t) \right) < \infty
$$
but which is not a Carleson measure for $H^2(\bbT^\infty)$.
\end{proposition}
\begin{proof}
If $\mu$ is a finite complex measure on $(-1,1)^\infty$, it induces a Helson matrix $M(\hel_\mu)$ by the formula \eqref{eq:mmtpoly}.
By Theorem~\ref{thm.d2}, if $\mu$ is a Carleson measure for $H^2(\bbT^\infty)$, then $M(\hel_\mu) \colon  \ell^2(\bbN) \to \ell^2(\bbN)$ is bounded.

Let $X$ be  the Banach space of complex measures on $(-1,1)^\infty$ which have finite norm
$$
\|\mu\|_X = \sup_{s \in (-1,1)^\infty, d\geq1} \left( \prod_{j=1}^d (1-s_j^2) \int_{(-1,1)^\infty} |K_{d}(t,s)|^2 \, d|\mu|(t)\right) < \infty.
$$
We leave it to the reader to verify that $X$ is complete under this norm. 

We prove the proposition by contradiction. 
Suppose it were the case that $|\mu|$ is a Carleson measure whenever $\mu$ has finite $X$-norm. 
The map $\mu \mapsto M(\hel_\mu)$ is then an everywhere defined linear map of $X$ to $B(\ell^2(\bbN))$, the Banach space of bounded linear operators on $\ell^2(\bbN)$. It is clear that the map $\mu \to M(\hel_\mu)$ is closed. Therefore, it follows from the closed graph theorem that there is an absolute constant $C < \infty$ such that
$$
\|M(\hel_\mu)\|_{B(\ell^2(\bbN))} \leq C\|\mu\|_X.
$$
We shall show by example that this is false.

Let $\nu$ be the measure on $(0,1)$ such that $d \nu(t) = \frac{t}{2}dt$. For $N \geq 1$, let $\mu_N$ be the measure 
$$\mu_N = \prod_{j=1}^N \nu \times \prod_{j=N+1}^\infty \delta_0$$
on $[0,1)^\infty$, where $\delta_0$ denotes the Dirac measure at $0$.

The moment sequence associated with $\mu_N$ is 
$$
\hel_N(n) 
= 
\int_{[0,1)^\infty} t^{\kappa(n)} \, d\mu_N(t) 
= 
\prod_{j=1}^N \frac{1}{2(\kappa_j + 2)} \prod_{j=N+1}^\infty \delta(\kappa_j), \qquad n = p^{\kappa} \geq 1.
$$
Let $\han(\kappa)=1/(2(\kappa+2))$; the norm of the Hankel matrix $H_1(\han)$ is $\|H_1(\han)\|_{B(\ell^2(\bbN))} = \pi/2$ \cite{Mag51}. 
Also, let $H_1(\delta)$ be the trivial Hankel matrix, corresponding to the Kronecker sequence $\delta$. The norm of $H_1(\delta)$ is $1$. 
Since $\hel_N$ is a multiplicative sequence satisfying $\{\hel_N(q^{j+k})\}_{j,k=0}^\infty = H_1(\beta)$ for the first $N$ primes $q$, and $\{\hel_N(q^{j+k})\}_{j,k=0}^\infty = H_1(\delta)$ for primes $q > p_N$, it follows that the Helson matrix $M(\hel_N)$ decomposes into the tensor product
$$
M(\hel_N) 
= 
\underbrace{H_1(\han) \otimes H_1(\han) \otimes \cdots \otimes H_1(\han)}_{N \; \mathrm{factors}} \otimes H_1(\delta) \otimes H_1(\delta) \otimes \cdots. 
$$
We conclude that
$$\|M(\hel_N)\|_{B(\ell^2(\bbN))} = \left(\frac{\pi}{2}\right)^N,$$
see \cite[Lemma 2]{BP15}. On the other hand, a straightforward calculation shows that
$$\int_0^1 \frac{1}{(1-st)^2} \, d\nu(t) = \frac{1}{2s^2}\left(\log(1-s) + \frac{1}{1-s} - 1\right) \leq \frac{1}{1-s^2} \quad s \in (0,1),$$
from which it is clear that $\|\mu_N\|_X = 1$. Hence
$$\lim_{N \to \infty} \frac{\|M(\hel_N)\|_{B(\ell^2(\bbN))}}{\|\mu_N\|_X} = \infty,$$
finishing the proof. 
\end{proof}

\bibliographystyle{amsplain} 
\bibliography{moment2} 

\def\cprime{$'$}
\providecommand{\bysame}{\leavevmode\hbox to3em{\hrulefill}\thinspace}
\providecommand{\MR}{\relax\ifhmode\unskip\space\fi MR }
\providecommand{\MRhref}[2]{%
  \href{http://www.ams.org/mathscinet-getitem?mr=#1}{#2}
}
\providecommand{\href}[2]{#2}
\begin{thebibliography}{10}

\bibitem{Ber68}
Ju.~M. Berezanski{\u\i}, \emph{Expansions in eigenfunctions of selfadjoint
  operators}, Translated from the Russian by R. Bolstein, J. M. Danskin, J.
  Rovnyak and L. Shulman. Translations of Mathematical Monographs, Vol. 17,
  American Mathematical Society, Providence, R.I., 1968.

\bibitem{BCR}
Christian Berg, Jens Peter~Reus Christensen, and Paul Ressel, \emph{Harmonic
  analysis on semigroups}, Graduate Texts in Mathematics, vol. 100,
  Springer-Verlag, New York, 1984, Theory of positive definite and related
  functions.

\bibitem{BP15}
Ole~Fredrik Brevig and Karl-Mikael Perfekt, \emph{Failure of {N}ehari's theorem
  for multiplicative {H}ankel forms in {S}chatten classes}, Studia Math.
  \textbf{228} (2015), no.~2, 101--108.

\bibitem{BPSVolterra}
Ole~Fredrik Brevig, Karl-Mikael Perfekt, and Kristian Seip, \emph{Volterra
  operators on {H}ardy spaces of {D}irichlet series}, J. Reine Angew. Math., in
  press.

\bibitem{BPSSV}
Ole~Fredrik Brevig, Karl-Mikael Perfekt, Kristian Seip, Aristomenis~G.
  Siskakis, and Dragan Vukoti{\'c}, \emph{The multiplicative {Hilbert} matrix},
  Adv. Math. \textbf{302} (2016), 410--432.

\bibitem{Carleson}
Lennart Carleson, \emph{Interpolations by bounded analytic functions and the
  corona problem}, Ann. of Math. (2) \textbf{76} (1962), 547--559.

\bibitem{Chang}
Sun-Yung~A. Chang, \emph{Carleson measure on the bi-disc}, Ann. of Math. (2)
  \textbf{109} (1979), no.~3, 613--620.

\bibitem{CG86}
Brian~J. Cole and T.~W. Gamelin, \emph{Representing measures and {H}ardy spaces
  for the infinite polydisk algebra}, Proc. London Math. Soc. (3) \textbf{53}
  (1986), no.~1, 112--142.

\bibitem{Dev57}
A.~Devinatz, \emph{Two parameter moment problems}, Duke Math. J. \textbf{24}
  (1957), 481--498.

\bibitem{FergusonLacey}
Sarah~H. Ferguson and Michael~T. Lacey, \emph{A characterization of product
  {BMO} by commutators}, Acta Math. \textbf{189} (2002), no.~2, 143--160.

\bibitem{Gu}
Caixing Gu, \emph{Finite rank {H}ankel operators on the polydisk}, Linear
  Algebra Appl. \textbf{288} (1999), no.~1-3, 269--281.

\bibitem{HLS}
H{\aa}kan Hedenmalm, Peter Lindqvist, and Kristian Seip, \emph{A {H}ilbert
  space of {D}irichlet series and systems of dilated functions in
  {$L^2(0,1)$}}, Duke Math. J. \textbf{86} (1997), no.~1, 1--37.

\bibitem{Helson06}
Henry Helson, \emph{Hankel forms and sums of random variables}, Studia Math.
  \textbf{176} (2006), no.~1, 85--92.

\bibitem{Helson10}
\bysame, \emph{Hankel forms}, Studia Math. \textbf{198} (2010), no.~1, 79--84.

\bibitem{Infusino}
Maria Infusino, \emph{Quasi-analyticity and determinacy of the full moment
  problem from finite to infinite dimensions}, Stochastic and Infinite
  Dimensional Analysis. Trends in Mathematics. (Bernido C., Carpio-Bernido M.,
  Grothaus M., Kuna T., Oliveira M., and da~Silva~J., eds.), Birkh\"auser,
  Cham, 2016.

\bibitem{JPR}
Svante Janson, Jaak Peetre, and Richard Rochberg, \emph{Hankel forms and the
  {F}ock space}, Rev. Mat. Iberoamericana \textbf{3} (1987), no.~1, 61--138.

\bibitem{LaceyTerwilleger}
Michael Lacey and Erin Terwilleger, \emph{Hankel operators in several complex
  variables and product {BMO}}, Houston J. Math. \textbf{35} (2009), no.~1,
  159--183.

\bibitem{Mag51}
Wilhelm Magnus, \emph{\"{U}ber einige beschr\"ankte {M}atrizen}, Arch. Math.
  \textbf{2} (1949--1950), 405--412 (1951).

\bibitem{Nehari}
Zeev Nehari, \emph{On bounded bilinear forms}, Ann. of Math. (2) \textbf{65}
  (1957), 153--162.

\bibitem{Olsen}
Jan-Fredrik Olsen and Eero Saksman, \emph{On the boundary behaviour of the
  {H}ardy spaces of {D}irichlet series and a frame bound estimate}, J. Reine
  Angew. Math. \textbf{663} (2012), 33--66.

\bibitem{OS}
Joaquim Ortega-Cerd{\`a} and Kristian Seip, \emph{A lower bound in {N}ehari's
  theorem on the polydisc}, J. Anal. Math. \textbf{118} (2012), no.~1,
  339--342.

\bibitem{Peller}
Vladimir~V. Peller, \emph{Hankel operators and their applications}, Springer
  Monographs in Mathematics, Springer-Verlag, New York, 2003.

\bibitem{PP2}
Karl-Mikael Perfekt and Alexander Pushnitski, \emph{On the spectrum of the
  multiplicative {H}ilbert matrix}, Arkiv f\"or Matematik, to appear.

\bibitem{Pow82}
S.~C. Power, \emph{Finite rank multivariable {H}ankel forms}, Linear Algebra
  Appl. \textbf{48} (1982), 237--244.

\bibitem{PutinarSchmudgen}
Mihai Putinar and Konrad Schm{\"u}dgen, \emph{Multivariate determinateness},
  Indiana Univ. Math. J. \textbf{57} (2008), no.~6, 2931--2968.

\bibitem{Queffelec}
Herv{\'e} Queff{\'e}lec and Martine Queff{\'e}lec, \emph{Diophantine
  approximation and {D}irichlet series}, Harish-Chandra Research Institute
  Lecture Notes, vol.~2, Hindustan Book Agency, New Delhi, 2013.

\bibitem{Rochberg95}
Richard Rochberg, \emph{A {K}ronecker theorem for higher order {H}ankel forms},
  Proc. Amer. Math. Soc. \textbf{123} (1995), no.~10, 3113--3118.

\bibitem{Rosenblum1}
Marvin Rosenblum, \emph{On the {H}ilbert matrix. {I}}, Proc. Amer. Math. Soc.
  \textbf{9} (1958), 137--140.

\bibitem{Rosenblum2}
\bysame, \emph{On the {H}ilbert matrix. {II}}, Proc. Amer. Math. Soc.
  \textbf{9} (1958), 581--585.

\bibitem{Samo91}
Yu.~S. Samo{\u\i}lenko, \emph{Spectral theory of families of selfadjoint
  operators}, Mathematics and its Applications (Soviet Series), vol.~57, Kluwer
  Academic Publishers Group, Dordrecht, 1991, Translated from the Russian by E.
  V. Tisjachnij.

\bibitem{Sim98}
Barry Simon, \emph{The classical moment problem as a self-adjoint finite
  difference operator}, Adv. Math. \textbf{137} (1998), no.~1, 82--203.

\bibitem{Vasilescu}
F.-H. Vasilescu, \emph{Hamburger and {S}tieltjes moment problems in several
  variables}, Trans. Amer. Math. Soc. \textbf{354} (2002), no.~3, 1265--1278
  (electronic).

\bibitem{Widom66}
Harold Widom, \emph{Hankel matrices}, Trans. Amer. Math. Soc. \textbf{121}
  (1966), 1--35.

\end{thebibliography}

\end{document}